\documentclass{amsart}

\usepackage{amssymb, amsmath,mathtools}
\usepackage{mathrsfs}
\usepackage{amscd}
\usepackage{verbatim}
\usepackage{stmaryrd}

\usepackage{enumerate}

\allowdisplaybreaks
\usepackage[colorlinks,linkcolor={blue},citecolor={blue},urlcolor={red},]{hyperref}

\usepackage[colorinlistoftodos,prependcaption,textsize=tiny]{todonotes}

\usepackage{color}

\theoremstyle{plain}
\newtheorem{theorem}{Theorem}[section]
\theoremstyle{remark}
\newtheorem{remark}[theorem]{Remark}
\newtheorem{example}[theorem]{Example}
\theoremstyle{plain}
\newtheorem{corollary}[theorem]{Corollary}
\newtheorem{lemma}[theorem]{Lemma}
\newtheorem{proposition}[theorem]{Proposition}
\newtheorem{definition}[theorem]{Definition}

\newtheorem{problem}[theorem]{Problem}
\numberwithin{equation}{section}

\def\N{{\mathbb N}}
\def\Z{{\mathbb Z}}

\def\R{{\mathbb R}}
\def\C{{\mathbb C}}

\newcommand{\E}{{\mathbb E}}
\renewcommand{\P}{{\mathbb P}}
\newcommand{\mF}{{\mathbb F}}
\newcommand{\F}{{\mathscr F}}

\newcommand{\A}{{\mathcal A}}

\renewcommand{\H}{{\mathbb H}}

\newcommand{\g}{\gamma}
\renewcommand{\a}{\kappa}

\newcommand{\om}{\omega}
\renewcommand{\O}{\Omega}

\newcommand{\angH}{\omega_{H^{\infty}}}

\renewcommand{\Re}{\hbox{\rm Re}\,}

\newcommand{\D}{{\mathscr D}}
\newcommand{\calL}{{\mathscr L}}
\newcommand{\Sc}{{\mathscr S}}

\newcommand{\one}{{{\bf 1}}}

\newcommand{\dps}{\displaystyle}

\newcommand{\wh}{\widehat}
\newcommand{\supp}{\text{\rm supp\,}}

\newcommand{\Four}{\mathcal{F}}

\newcommand{\wt}{\widetilde}
\renewcommand{\tilde}{\widetilde}

\newcommand{\Tr}{\mathtt{Tr}_0}

\renewcommand{\S}{\mathscr{S}}
\newcommand{\Id}{I}

\newcommand{\Sch}{\mathscr{S}}

\newcommand{\Ext}{\mathsf{E}_I}

\newcommand{\TD}{{\mathscr S'}}

\newcommand{\loc}{\rm loc}

\def\Xint#1{\mathchoice
   {\XXint\displaystyle\textstyle{#1}}%
   {\XXint\textstyle\scriptstyle{#1}}%
   {\XXint\scriptstyle\scriptscriptstyle{#1}}%
   {\XXint\scriptscriptstyle\scriptscriptstyle{#1}}%
   \!\int}
\def\XXint#1#2#3{{\setbox0=\hbox{$#1{#2#3}{\int}$}
     \vcenter{\hbox{$#2#3$}}\kern-.5\wd0}}

\newcommand{\Do}{{\mathsf{D}}}
\newcommand{\Ran}{{\mathsf{R}}}

\newcommand{\Dd}{\dot{\mathsf{D}}}

\newcommand{\norm}[1]{{\left\vert\kern-0.25ex\left\vert\kern-0.25ex\left\vert #1
    \right\vert\kern-0.25ex\right\vert\kern-0.25ex\right\vert}}

\newcommand{\Wz}{\prescript{}{0}{W}}

\newcommand{\Hz}{\prescript{}{0}{H}}

\newcommand{\M}{\mathfrak{T}}
\newcommand{\Co}{\mathscr{C}}
\newcommand{\Fsigma}{\mathfrak{P}}

\newcommand{\Der}{\mathfrak{D}}
\newcommand{\Dert}{\tilde{\mathfrak{D}}}

\newcommand{\Restr}{\mathsf{R}}
\newcommand{\Restrepsilon}{\Restr^{\varepsilon}}
\newcommand{\Restrkepsilon}{\Restr_{k}^{\varepsilon}}

\begin{document}

\title[Trace embedding and its applications]{On the trace embedding and its applications to evolution equations}

\author{Antonio Agresti}
\address{Institute of Science and Technology Austria (IST Austria)\\
Am Campus 1\\
3400 Klosterneuburg\\
Austria}
\email{antonio.agresti92@gmail.com}

\author{Nick Lindemulder}
\address{Institute of Analysis \\
Karlsruhe Institute of Technology \\
Englerstraße 2 \\
76131 Karlsruhe\\
Germany}
\email{n.lindemulder@hotmail.com}

\author{Mark Veraar}
\address{Delft Institute of Applied Mathematics\\
Delft University of Technology \\ P.O. Box 5031\\ 2600 GA Delft\\The
Netherlands} \email{M.C.Veraar@tudelft.nl}



\thanks{
The first author has been partially supported by the Nachwuchsring -- Network for the promotion of young scientists -- at TU Kaiserslautern.
The second and third author were supported by the Vidi subsidy 639.032.427 of the Netherlands Organisation for Scientific Research (NWO)}

\date{\today}

\begin{abstract}
In this paper we consider traces at initial times for functions with mixed time-space smoothness. Such results are often needed in the theory of evolution equations. Our result extends and unifies many previous results. Our main improvement is that we can allow general interpolation couples. The abstract results are applied to regularity problems for fractional evolution equations and stochastic evolution equation, where uniform trace estimates on the half-line are shown.
\end{abstract}

\maketitle

\section{Introduction}

In the theory of evolution equations one often needs to describe the precise regularity of the trace of $u$ at $t=0$, where $u$ is a function which has mixed fractional smoothness in time and space. Parabolic regularity usually implies that $u$ is in
\begin{equation}\label{eq:intersspace}
 \mathcal{A}^{s,p}(\R_+;X_0)\cap L^p(\R_+;X_1),
\end{equation}
where $\mathcal{A}^{s,p}=W^{s,p}$ is the Sobolev–Slobodetskii space, or $\mathcal{A}=H^{s,p}$ is the Bessel potential space, or $\mathcal{A}^{s,p} = F^{s}_{p,q}$ is the Triebel-Lizorkin space, and $s>1/p$. Typically, $X_0 = L^q$ and $X_1$ is some Sobolev–Slobodetskii space as well. In the case $s\in \N$ is an integer, and $\mathcal{A}$ is a Sobolev space, then the precise regularity is well-understood and related to the so-called trace method for real interpolation due to J.L. Lions (see \cite{Tr78} for a detailed account and references).

For $s\in (0,\infty)\setminus \N$, the situation is more delicate and much less is known. This case occurs naturally in:
\begin{enumerate}[{\rm (a)}]
\item\label{it:EEvol} Evolution equations of Volterra type;
\item\label{it:SE} Stochastic evolution equations driven by Brownian noise.
\end{enumerate}
In \eqref{it:EEvol} the smoothness in time is fractional because of the fractional order time derivative in the equations. In \eqref{it:SE} the smoothness in time is fractional because Brownian motion is only $C^\alpha$-regular for all $\alpha\in [0,1/2)$.
These two examples will be discussed in more detail in Section \ref{s:applications}.

Fractional smoothness also occurs when dealing with inhomogeneous boundary value problems. We illustrate this with a simple but nontrivial heat equation. On the time intervals $J = (0,T)$ and the half space $\R^d_+ = (0,\infty)\times\R^{d-1}$ consider
\begin{equation*}
\left\{
  \begin{array}{ll}
    \partial_t u  - \Delta u = 0, &  \text{on $J\times \R^d_+$;} \\
    u=g,  & \text{on $J\times \R^{d-1}$;} \\
    u = 0, & \hbox{on $\{0\}\times\R^d_+$.}
  \end{array}
\right.
\end{equation*}
Typically one wants to characterize when
\begin{equation}\label{eq:uMRspace}
u\in W^{1,p}(J;L^p(\R^d_+))\cap L^p(J;W^{2,p}(\R^d_+)).
\end{equation}
By classical parabolic regularity theory (see \cite{DHPinh, LSU}), the latter holds if and only if
\begin{align*}
g \in \mathbb{B}:=W^{1-\frac1{2p},p}(J; L^p(\R^{d-1}))\cap L^p(J;W^{2-\frac1p,p}(\R^{d-1})),
\end{align*}
and the compatibility conditions $g =0$ holds at $\{0\}\times\R^{d-1}$ if $p>3/2$ and no condition is required if $p<3/2$ (we avoid the case $p=3/2$). The latter space fits exactly in the setting \eqref{eq:intersspace}.
In order to prove the above characterization it is important to characterize the trace space of $\mathbb{B}$ at $t=0$, which is known to be $W^{2-\frac3p,p}(\R^{d-1})$. In case one wants $L^q$-integrability in space in \eqref{eq:uMRspace}, then the space for $g$ becomes more involved and Triebel-Lizorkin spaces are needed (see \cite{DHPinh}).

Without any structured theory, the complexity of the spaces involved easily gets out of hand for more difficult boundary value problems. We refer to \cite{MV14} for the application to the inhomogeneous Stefan problem with Gibbs-Thomson correction, where many traces spaces need to be characterized for spaces of the form \eqref{eq:intersspace}, where often even three spaces are intersected. The same technique can be used to include inhomogeneities in other equations of so called {\em mixed order type}. Details on equations of mixed order type can be found in \cite{DenkKaip, DenkSaalSeiler}.

After these motivations, we return to the general setting of the trace problem for \eqref{eq:intersspace}. There are actually two problems to be solved if one wants to characterize the unknown trace space $X_{\Tr}$. In the following $\Tr u:=u(0)$ and $\R_+:=(0,\infty)$.

\begin{problem}[Trace problem]\label{pro:trace}
Let $s>0$, $p\in (1, \infty)$, $\gamma\geq 0$. Find $X_{\Tr}$ such that
\begin{enumerate}[{\rm (1)}]
\item\label{it:tracech1} $\Tr:\mathcal{A}^{s,p}(\R_+,t^{\gamma}dt;X_0)\cap L^p(\R_+,t^{\gamma}dt;X_1)\to X_{\Tr}$ is  well-defined and bounded;
\item\label{it:tracech2} $X_{\Tr}$ is the smallest space for which \eqref{it:tracech1} holds.
\end{enumerate}
\end{problem}
A sufficient condition for \eqref{it:tracech2} is
\begin{enumerate}[{\rm (2)'}]
\item There exists a right inverse to $\Tr$.
\end{enumerate}

In the Sobolev–Slobodetskii scale ($\mathcal{A} = W$) results about traces of spaces with mixed order (fractional) smoothness can be found in \cite{DiBlasio84} and extensions to weighted spaces in \cite[Theorem 4.2]{MeySchn12}. The case of Bessel-potential spaces ($\mathcal{A}=H$) was considered in special cases in \cite[Theorem 3.6]{Zacher05} under geometric restrictions on $X_0$ and under the condition that $X_1$ is the domain of an $R$-sectorial operator $T$. In the paper \cite{MV14} by Meyries and the third named author, it was shown that \cite{DiBlasio84,MeySchn12,Zacher05} can be unified and geometric and $R$-sectoriality conditions can be omitted. However, in \cite{MV14} it is still assumed that $X_1$ is domain of some sectorial operator $T$, and there is an angle condition on the sectoriality of $T$, which is in terms of the smoothness $s$.

In the current paper we will show that one can actually consider an arbitrary interpolation couple $(X_0,X_1)$, and this extends the boundedness part of \cite{MV14}, which is an important step in the solution to Problem \ref{pro:trace}:

\begin{theorem}\label{thmintro:t:trace_FF_FL_spaces;H}
Let $(X_0, X_1)$ be an interpolation couple of Banach spaces. Let $p\in (1, \infty)$, $\g \in (-1,p-1)$, $q\in [1, \infty]$, and $s>0$. Let $\mathcal{A}^{s,p}\in \{H^{s,p}, W^{s,p}, F^{s}_{p,q}\}$.
Then for $s>\frac{1+\g}{p}$ the following mapping is bounded
\begin{align*}
\Tr: \mathcal{A}^{s,p}(\R_+,t^{\gamma}dt;X_0)\cap L^{p}(\R_+,t^{\gamma}dt;X_1) &\to (X_0,X_1)_{1 - \frac{1+\gamma}{s p},p}, \ \ \Tr u = u(0).
\end{align*}
Moreover, if $\gamma\in [0,p-1)$, then $\mathcal{A}^{s,p}(\R_+,t^{\gamma}dt;X_0)\cap L^p(\R_+,t^{\gamma}dt;X_1)$ continuously embeds into
\begin{align*}
& C_{\mathrm{b}}\big([0,\infty);(X_0,X_1)_{1 - \frac{1+\gamma}{s p},p}\big) \ \ \ \ \  \ \ \ \text{if $s>\frac{1+\g}{p}$,}
\\ & C_{{\rm{b}},\gamma/p}\big((0,\infty);(X_0,X_1)_{1 - \frac{1}{s p},p}\big)  \ \ \ \  \ \text{if $s>\frac{1}{p}$.}
\end{align*}
\end{theorem}
Here $C_{\rm{b}}([0,\infty);X)$ and $C_{\rm{b},\eta}((0,\infty);X)$ denote the spaces of $X$-valued continuous functions on $[0,\infty)$ and $(0,\infty)$ for which $\dps \sup_{t\in [0,\infty)}\|u(t)\|_X< \infty$ and $\dps \sup_{t>0}t^{\eta} \|u(t)\|_X<\infty$, respectively.

The above result extends and unifies many of the previous mentioned results.
Moreover, consequences with homogeneous norms are derived in Sections \ref{sec:homF} and \ref{s:traceBessel}. Optimality holds in many cases (see \cite[Theorem 1.1]{MV14}), but not in general. For instance in the trivial case where $X_0\hookrightarrow X_1$, the trace mapping is actually bounded with values in the smallest space $X_0$, which shows that optimality is a delicate problem. Further results on this will be given in a subsequent paper.

Both cases $H^{s,p}$  and $W^{s,p}$ of Theorem \ref{thmintro:t:trace_FF_FL_spaces;H} follow from the case $F^{s}_{p,q}$. Indeed, these spaces are sandwiched between the smallest case $q=1$, and largest case $q=\infty$ of $F^{s}_{p,q}$ (see \eqref{eq:EmbElementary} and \eqref{eq:equivhom} below), and the image space $(X_0,X_1)_{\theta,p}$ is independent of $q$.
The result on the weighted continuous functions  with values in $(X_0,X_1)_{1 - \frac{1}{s p},p}$ does not appear elsewhere as far as we know. If $\gamma>0$ this result provides extra information on the regularity in space for $t>0$.
Theorem \ref{thmintro:t:trace_FF_FL_spaces;H} is proved in Sections \ref{s:trace} and \ref{s:traceBessel} where we even allow different $p$, $q$, and $\gamma$ for both spaces in the intersection.

Theorem \ref{thmintro:t:trace_FF_FL_spaces;H} shows that the angle conditions of \cite[Corollary 4.7]{MV14} are not needed for the boundedness of the trace mapping. This angle condition led to restrictions in \cite[Section 5]{Pruss19} which can now be omitted as will be explained in detail in Section \ref{s:applications}. Another important aspect is that, unlike in the previous theory, we can allow $X_1$ to be the homogeneous domain of a sectorial operator. For instance our setting allows to take $X_0 = L^q(\R^d)$ and the homogeneous space $X_1 = \dot{W}^{m,q}(\R^d)$ which is important for certain optimal regularity results for (non)linear evolution equation (cf.\ \cite{danchin2020free} and references therein). In particular, we will state global estimates on $\R_+$ by using homogeneous function spaces, which would be impossible if one is limited to the inhomogeneous setting. These things are demonstrated in Section \ref{s:applications}.

\medskip

Overview:
\begin{itemize}
\item In Section \ref{sec:Prel} we present some preliminaries on function spaces on $\R^d$, and on sectorial operators and functional calculus.
\item In Section \ref{sec:Halfline} we extend some well-known results to function spaces on the half line.
\item In Section \ref{s:trace} we prove the main boundedness result concerning the trace operator.
\item Section \ref{s:traceBessel} we briefly discuss the consequence for Bessel potential spaces.
\item In Section \ref{s:applications} we show how Theorem \ref{thmintro:t:trace_FF_FL_spaces;H} can be used to obtain a priori estimates on infinite time intervals $[0,\infty)$ for Volterra equations and stochastic evolution equations.
\end{itemize}

\subsubsection*{Notation}
The following standard notations are used frequently in the paper $\N_0:=\N\cup\{0\}$, $\R_+ = (0,\infty)$, $w_{\gamma}(t) = t^{\gamma}$. The real interpolation space $(X_0, X_1)_{\theta,p}$ for $\theta\in (0,1)$ and $p\in [1, \infty]$. Dense and continuous embedding $X \stackrel{d}{\hookrightarrow} Y$.

$B^{s}_{p,q}$ Besov space, $F^{s}_{p,q}$ Triebel-Lizorkin space, $H^{s,p}$ Bessel potential space, $W^{m,p}$ Sobolev space, $W^{s,p}$ Sobolev-Slobodetskii space.
Difference seminorm $[f]_{F^{s}_{q,p}(I,w_{\g};X)}$ (see Remark \ref{rem:seminorms}). $\Tr^k$ is the trace mapping and is defined in Lemma \ref{lemma:t:trace_FF_FL_spaces;X_0}.
$C_{\rm{b},\mu}(\R_+;X)$ see \eqref{eq:Cbmu}.
$\Der u = u'$ and $\Dert u = -u'$ see \eqref{eq:defC}-\eqref{eq:defB}.

$A \eqsim B$ means that there is a constant $C>0$ such that $C^{-1} A \leq B \leq C A$. Here the constant $C$ is typically only dependent of parameters which are clear in each context.

\section{Preliminaries}\label{sec:Prel}
In this section we collect definitions and basic results for function spaces with power weights.

Let $\Sc(\R^d;X)$ denote the Schwartz functions with values in $X$ with its usual topology, and let $\Sc'(\R^d;X) = \calL(\Sc(\R^d),X)$ denote the $X$-valued tempered distributions. For $\Omega\subseteq \R^d$ open, let $\D(\Omega) = C^\infty_c(\Omega)$ be the $C^\infty$-functions which have compact support in $\Omega$ endowed with its usual topology, and let $\D'(\Omega;X) =\calL(\D(\Omega),X)$  denote the $X$-valued distributions.

\subsection{Real interpolation and the mean method}
In this subsection we collect basic facts on interpolation theory which will be needed in the paper. For an exposition of the general theory see \cite{BeLo}. As usual, we say that $(X_0,X_1)$ is an interpolation couple of Banach spaces if $X_0$ and $X_1$ are Banach spaces and there exists a topological Hausdorff space $V$, such that $X_i\hookrightarrow V$ continuously for $i\in \{0,1\}$.

Let $(X_0,X_1)$ be an interpolation couple of Banach spaces.  For the definition and the basic properties of the real interpolation spaces $(X_0, X_1)_{\theta,p}$ for $\theta\in (0,1)$ and $p\in [1, \infty]$,  we refer to \cite{BeLo, Analysis1, Tr78}. Recall that $X_0\cap X_1\hookrightarrow (X_0,X_1)_{\theta,p}\hookrightarrow X_0+X_1$, where the first embedding is dense if $p<\infty$. We will frequently use the following reiteration result:
\begin{equation}\label{eq:reit}
((X_0,X_1)_{\theta_0,p_0},(X_0,X_1)_{\theta_1,p_1})_{\alpha,p}=(X_0,X_1)_{\theta,p},
\end{equation}
where $p_0, p_1,p\in [1, \infty]$, $\alpha, \theta_0, \theta_1\in (0,1)$ and $\theta = (1-\alpha) \theta_0 + \alpha \theta_1$.

In this paper we will use the following alternative characterization of real interpolation spaces, which is a variant of the first mean method and can be found in \cite[Remark 2, p. 35]{Tr78}.

\begin{lemma}
\label{l:eqreal}
Let $\xi_0,\xi_1$ be real numbers such that $\xi_0 \xi_1<0$ and $1 \leq p_0,p_1<\infty$. Moreover, set
\begin{equation*}
\theta:= \frac{\xi_0}{\xi_0-\xi_1}, \qquad \frac{1}{p}:=\frac{1-\theta}{p_0}+\frac{\theta}{p_1}.
\end{equation*}
Then $(X_0,X_1)_{\theta,p}$ coincides with the set of all $x\in X_0+X_1$ for which there exist measurable functions $u_j:\R_+\to X_j$ (for $j\in \{0,1\}$) such that the maps $t \mapsto t^{\xi_j}u_j(t)$ belong to $ L^{p_j}(\R_+,\frac{dt}{t};X_j)$ and
\begin{equation}
\label{eq:xid}
x=u_0(t)+u_1(t),\qquad \text{for almost all}\;t>0.
\end{equation}
Furthermore,
$$
\|x\|_{(X_0,X_1)_{\theta,p}}\eqsim_{\theta,p,\xi_0,\xi_1} \inf_{u_0,u_1} \sum_{j\in \{0,1\}} \|t \mapsto t^{\xi_j}u_j(t) \|_{ L^{p_j}(\R_+,\frac{dt}{t};X_j)},
$$
where the infimum is taken over all $u_0,u_1$ such that \eqref{eq:xid} holds.
\end{lemma}

\begin{proof}
The proof is a simple modification of \cite[Remark 2, p. 35]{Tr78}. For convenience of the reader we provide the details. To begin, we recall that $(X_0,X_1)_{\theta,p}$ can be characterized as the set of all $x\in X_0+X_1$ such that there exist measurable functions $u_j:\R_+\to X_j$ (for $j\in \{0,1\}$) such that the maps $t \mapsto t^{j-\theta} u_j(t)$ ($j\in \{0,1\}$) belong to $ L^{p_j}(\R_+,\frac{dt}{t};X_j)$ and $x=u_0(t)+u_1(t)$ for almost all $t>0$. Furthermore,
$$
\|x\|_{(X_0,X_1)_{\theta,p}}\eqsim_{\theta,p} \inf_{u} \sum_{j\in \{0,1\}} \|t \mapsto t^{j-\theta}u_j(t) \|_{ L^{p_j}(\R_+,\frac{dt}{t};X_j)};
$$
where the infimum is taken over all $u$ such that \eqref{eq:xid} holds. See for instance, \cite[Theorem 1.5.2 a) p. 33]{Tr78} or \cite[Proposition C.3.7]{Analysis1}.

To prove the equivalence, it is enough to note that for any measurable maps $u_j$ ($j\in \{0,1\}$) such that $t \mapsto t^{\xi_j}u_j(t)\in L^{p_j}(\R_+,\frac{dt}{t};X_j)$, the functions $v_j(t):=u_j(t^{\frac{1}{\xi_1-\xi_0}})$ satisfy
$$
\int_0^{\infty} (t^{\xi_j} \|u(t)\|_{X_j})^{p_j} \frac{dt}{t}=\frac{1}{|\xi_1-\xi_0|} \int_0^{\infty} (\tau^{j-\theta} \|v_j(t)\|_{X_j})^{p_j} \frac{dt}{t},\quad j\in \{0,1\};
 $$
here we have used the transformation $t=\tau^{\frac{1}{\xi_1-\xi_0}}$ and that $\frac{dt}{t}=\frac{1}{\xi_1-\xi_0}\frac{d\tau}{\tau}$. Moreover,
 $$
 x=\sum_{j\in \{0,1\}} u_j( t^{\frac{1}{\xi_1-\xi_0}})=\sum_{j\in \{0,1\}} v_j( t), \qquad \text{for almost all}\;\; t\in\R_+.
 $$
Thus, the maps $u_j \mapsto v_j$ are isomorphisms and this concludes the proof.
\end{proof}

\subsection{Weighted function spaces}

In this subsection $X$ denotes a Banach space.
Let $p\in (1,\infty)$. Let $\Omega\subseteq \R^d$ be an open set. A weight $w:\Omega\to [0,\infty)$ is a measurable function which is nonzero a.e. The norm of $L^p(\Omega,w;X)$ is given by
\[\|f\|_{L^p(\Omega,w;X)} = \Big( \int_{\Omega} \|f(x)\|_X^p w(x) \, dx\Big)^{1/p}.\]
In case $X = \C$ we write $L^p(\Omega,w) = L^p(\Omega,w;\C)$, and in case $w = 1$ we write $L^p(\Omega;X) = L^p(\Omega,1;X)$. We say $w\in A_p$ (Muckenhoupt $A_p$ class) if
$$\sup_{Q\text{ cubes in }\R^d} \Big( \Xint-_Q w(x)\,dx\Big)\Big( \Xint-_Q w(x)^{-1/(p-1)}\,dx\Big)^{p-1} < \infty.$$

One further sets $A_\infty = \bigcup_{p> 1} A_p$. For details on $A_p$-weights we refer to \cite[Chapter 9]{GraModern} and \cite[Chapter V]{Stein93}.
We will mostly be using $d=1$ and the weight
$w_\gamma(t) = |t|^\gamma$ which is in $A_p$ if and only if $\gamma \in (-1,p-1)$ (see \cite[Example 9.1.7]{GraModern}).

Let $m\in \N_0$, $p\in [1, \infty]$ and a weight $w:\Omega\to X$ such that $w^{-1/(p-1)}\in L^1_{{\rm loc}}(\O)$. Note that, by H\"{o}lder's inequality, the latter implies $L^p(\O,w;X)\hookrightarrow L^1_{{\rm loc}}(\O;X)$. Denote by $W^{m,p}(\Omega,w;X)$ the space of all $f\in L^p(\Omega,w;X)$ for which the distributional derivative $\partial^{\alpha} f$ exists in $L^p(\Omega;X)$ for all $|\alpha|\leq m$, and set
\begin{equation}
\label{eq:Weighted_Sobolev_spaces}
\|f\|_{W^{m,p}(\Omega,w;X)} = \sum_{|\alpha|\leq m} \|\partial^{\alpha} f\|_{L^p(\Omega,w;X)}.
\end{equation}

\subsection{Besov, Triebel-Lizorkin, and Bessel potential spaces}

We recall the definitions of weighted spaces of smooth functions. For details in the unweighted case see \cite{Amannbook2,Saw18,Tr78,Tri83} and in the weighted case \cite{Bui82,MV12}.

Let $\Phi(\R^d)$ be the set of all sequences $(\varphi_k)_{k\geq 0} \subseteq \Sch(\R^d)$ such that
\begin{align}\label{eq:defPhisequence}
\wh{\varphi}_0 = \wh{\varphi}, \qquad \wh{\varphi}_1(\xi) = \wh{\varphi}(\xi/2) - \wh{\varphi}(\xi), \qquad \wh{\varphi}_k(\xi) = \wh{\varphi}_1(2^{-k+1} \xi), \quad k\geq 2, \qquad \xi\in \R^d,
\end{align}
where the Fourier transform $\wh{\varphi}$ of the generating function $\varphi\in \Sch(\R^d)$ satisfies
\begin{equation}\label{gen-func}
 0\leq \wh{\varphi}(\xi)\leq 1, \quad  \xi\in \R^d, \qquad  \wh{\varphi}(\xi) = 1 \ \text{ if } \ |\xi|\leq 1, \qquad  \wh{\varphi}(\xi)=0 \ \text{ if } \ |\xi|\geq \frac32.
\end{equation}
For  $(\varphi_k)_{k\geq 0} \in \Phi(\R^d)$ and $f\in \TD(\R^d;X)$ we let
\begin{equation}
\label{eq:def_S_k}
S_k f = \varphi_k * f = \Four^{-1} ( \wh{\varphi}_k \wh{f}).
\end{equation}
Given $p \in [1,\infty)$, $q\in [1,\infty]$, $w\in A_\infty$ and $s \in \R$, for $f\in {\mathscr S}'(\R^d;X)$ we set
 \begin{align*}
\|f\|_{B_{p,q}^s (\R^d,w;X)} &= \Big\| \big( 2^{sk}S_k f\big)_{k\geq 0} \Big\|_{\ell^q(L^p(\R^d,w;X))}, \\
\|f\|_{F_{p,q}^s (\R^d,w;X)} & = \Big\| \big( 2^{sk}S_k f\big)_{k\ge 0} \Big\|_{L^p(\R^d,w;\ell^q(X))}.
\end{align*}
The Besov space $B_{p,q}^s (\R^d,w;X)$ and the Triebel-Lizorkin space $F_{p,q}^s (\R^d,w;X)$ are those spaces on which the respective extended norms are finite. They are all Banach spaces. Any other $(\psi_k)_{k\geq 0} \in \Phi(\R^d)$ leads to an equivalent norm on the $B$- and $F$-spaces. Note that $B_{p,p}^s = F_{p,p}^s$.
It is well known that the following continuous embeddings hold for $\A\in \{B,F\}$:
\begin{align}\label{eq:embeddingsSchDist}
\Sch(\R^d;X) \hookrightarrow \mathcal{A}^{s}_{p,q}(\R^d;X) \stackrel{d}{\hookrightarrow} \TD(\R^d;X).
\end{align}
The first embedding is also dense if $q<\infty$ (see also Lemma \ref{l:density} below).
For $s\in \R\setminus \N_0$, we define the Sobolev-Slobodetskii space by
\[W^{s,p}(\R^d,w;X) = B_{p,p}^s (\R^d,w;X).\]

Let $p \in (1,\infty)$, $w\in A_p$ and $s \in \R$.  For $f\in {\mathscr S}'(\R^d;X)$ we set
\begin{equation*}
\|f\|_{H^{s,p}(\R^d,w;X)} = \Big\|\Four^{-1} [(1+|\cdot|^2)^{s/2} \Four(f) ]\Big\|_{L^p(\R^d,w;X)}
\end{equation*}
and we define the Bessel-potential space $H^{s,p}(\R^d,w;X)$ as the space on which this extended norm is finite.  Recall the following crucial embedding results for $\mathcal{A}\in \{H,W\}$, $p\in (1,\infty)$, $w\in A_p$ and $s \in \R$:
\begin{align}
\label{eq:EmbElementary}
F^{s}_{p,1}(\R^d,w;X) &\hookrightarrow \mathcal{A}^{s,p}(\R^d,w;X) \hookrightarrow F^{s}_{p,\infty}(\R^d,w;X).
\end{align}

The following simple density result will be needed.
\begin{lemma}
\label{l:density}
Let $(X_0,X_1)$ be an interpolation couple of Banach spaces. Let $p_0,p_1\in [1,\infty)$, $q_0, q_1\in [1, \infty)$, $s_0,s_1\in\R$, $w_0,w_1\in A_\infty$ and $\A\in \{F,B\}$. Then
\[\S(\R,X_0\cap X_1) \stackrel{d}{\hookrightarrow}
\mathcal{A}^{s_0}_{p_0,q_0}(\R,w_{0};X_0)\cap \mathcal{A}^{s_1}_{p_1,q_1}(\R,w_{1};X_1).\]
The same is true if $\A_{p_1,q_1}^{s_1}$ is replaced by $L^{p_1}$.
\end{lemma}

A similar result holds for $\mathcal{A}_{p_i,q_i}^{s_i}$ replaced by $H^{s_i,p_i}$ for $i\in \{0,1\}$. Since it will be not needed here, we do not include the details.

\begin{proof}
We provide some details in the case $\mathcal{A} =F$. The case of Besov spaces is similar.
For notational convenience we set $\mF:=F^{s_0}_{p_0,q_0}(\R,w_{\g_0};X_0)\cap F^{s_1}_{p_1,q_1}(\R,w_{\g_1};X_1)$.
We employ the argument of \cite[Lemma 3.8]{MV12}. Fix $f\in \mF$ and let $S_k$ be as in \eqref{eq:def_S_k}. Since $f_n:=\sum_{k=0}^n S_k f\to f$ as $n \to \infty$ in $F^{s_i}_{p_i,q_i}(\R,w_{\g_i};X_i)$ for each $i\in \{0,1\}$, it is enough to approximate $f_n$. Following \cite{MV12}, we fix $\eta\in \S(\R)$ such that $\eta(0)=1$ and $\supp(\wh{\eta})\subseteq \{|x|< 1\}$.
Reasoning as in \cite[Lemma 3.8]{MV12}, by \cite[Proposition 2.4]{MV12}, one has $\eta(\delta \cdot)f_n\in \S(\R,X_i)$ and $\eta(\delta \cdot )f_n\to f_n$ in  $F^{s_i}_{p_i,q_i}(\R,w_{\g_i};X_i)$ as $\delta \searrow 0$ for each $i\in \{0,1\}$. This concludes the proof in the case $\A=F$.

The case $\A_1=L^{p_1}$ follows by noticing that $(S_k)_{k\geq 0}$ converges strongly to the identity in $L^{p_1}(\R,w_{\g_1};X_1)$ and therefore the above proof holds also in this case.
\end{proof}

\subsection{Sectorial operators and functional calculus}
\label{sss:sectorial_op_fractional_calculus}
For a detailed discussion on the theory below we refer to \cite{Analysis2,Haase:2,KuWe,pruss2016moving}.

Let $(\varepsilon_j)_{j\geq 1}$ be a sequence such that $\varepsilon_j$'s are independent random variable on a probability space $(\O,\A,\P)$ such that $\P(\varepsilon_j=1)=\P(\varepsilon_j=-1)=1/2$ for all $j\geq 1$. Usually, such sequence is called a Rademacher sequence. A family $\M\subseteq \calL(X)$ is called $R$-bounded if there exists $C>0$ such that for all $N\geq 1$, $x_1,\dots,x_N\in X$ and $T_1,\dots,T_N\in \M$,
$$
\Big\|\sum_{j=1}^N \varepsilon_j T_j x_j\Big\|_{L^2({\O};X)}\leq C\Big\|\sum_{j=1}^N \varepsilon_j x_j\Big\|_{L^2({\O};X)}.
$$
The infimum of all such $C>0$ will be denoted by $R(\M)$. With a slight abuse of notation, we write $R(\M)<\infty$ if $\M$ is $R$-bounded.

A closed linear operator $A:\Do(A)\subseteq X\to X$ is said to be sectorial (resp. $R$-sectorial) if $\overline{\Ran(A)}=\overline{\Do(A)}=X$ and there exists $\phi\in (0,\pi)$ such that $\sigma(A)\subseteq \{z\in \C\,:\,|\arg z|<\phi\}=:\Sigma_{\phi}$ and $\sup_{\lambda\in \C\setminus \Sigma_{\phi}}\|\lambda (\lambda-A)^{-1}\|_{\calL(X)}<\infty$ (resp. $R(\lambda(\lambda-A)^{-1}\,:\,\lambda \in \C\setminus \Sigma_{\phi})<\infty$). Moreover, we denote $\om(A)$ (resp. $\om_R(A)$) the angle of sectoriality (resp. $R$-sectoriality) and it is the infimum over all $\phi \in (0,\pi)$ as above.

Next, we define the $H^{\infty}$-calculus. For $\varphi\in(0,\pi)$, we denote by $H^{\infty}_0(\Sigma_{\varphi})$ the set of all holomorphic function $f:\Sigma_{\varphi}\rightarrow \C$ such that $|f(z)|\leq C|z|^{\varepsilon}/(1+|z|^{2\varepsilon})$ for some $C,\varepsilon>0$ independent of $z\in \Sigma_{\varphi}$. Let $A$ be a sectorial operator of angle $\om(A)<\nu<\varphi$. Then for $f\in H^{\infty}_0(\Sigma_{\varphi})$ we set
\begin{equation}
\label{eq:dunford}
f(A):=\frac{1}{2\pi i} \int_{\partial\Sigma_{\nu}} f(z)R(z,A)\,dz;
\end{equation}
where the orientation of $\partial \Sigma_{\nu}$ is such that $\sigma(A)$ is on the right. By \cite[Section 10.2]{Analysis2}, $f(A)$ is well-defined in $\calL(X)$ and it is independent of $\nu\in (\om(A),\varphi)$.

Furthermore, the operator $A$ is said to have a bounded $H^{\infty}(\Sigma_{\varphi})$-calculus if there exists $C>0$ such that for all $f\in H^{\infty}_0(\Sigma_{\varphi})$,
\begin{equation*}
\|f(A)\|_{\calL(X)}\leq C\|f\|_{H^{\infty}(\Sigma_{\varphi})}\,,
\end{equation*}
where $\|f\|_{H^{\infty}(\Sigma_{\varphi})}=\sup_{z\in\Sigma_{\varphi}}|f(z)|$. Finally, $\angH(A)$ denotes the infimum of all $\varphi\in (\om(A),\pi)$ such that $A$ has a bounded $H^{\infty}(\Sigma_{\varphi})$-calculus.

\subsubsection*{The homogeneous fractional scale $(\Dd(A^{\sigma}):\sigma\in \R)$}
\label{sss:homogeneous_scale}
Let $A$ be a sectorial operator on $X$. For each $\sigma\in \R$, $A^{\sigma}$ defines a closed injective linear operator on $X$ with domain $\Do(A^{\sigma})$ (see e.g. \cite[Chapter 3]{Haase:2} or \cite[Subsection 3.3]{pruss2016moving}). Furthermore, we may define the following spaces (see \cite[Appendix]{KuWe})
\begin{equation}
\label{eq:def_homogeneous_fractional_powers}
\Dd(A^{\sigma}):=(\Do(A^{\sigma}),\|A^{\sigma}\cdot\|_{X})^{\sim}, \qquad \sigma\in \R,
\end{equation}
where $\sim$ denotes the completion. Note that, if $0\not\in \rho(A)$, then $\Dd(A^{\alpha})\not\hookrightarrow \Dd(A^{\beta})$ even if $\alpha>\beta$.  Moreover, for $\varphi\in \R$, $p\in (1,\infty)$ and $ \N\ni k>|\varphi|$ we set
\begin{equation}
\label{eq:def_Dd_A}
\Dd_A(\varphi,p):=(\Dd(A^{-k}),\Dd(A^k))_{\theta,p},   \ \ \ \theta:=1/2+\varphi/(2k).
\end{equation}
By \eqref{eq:reit}, for $\theta_0,\theta_1,\phi\in (0,1)$ and $q_0,q_1,q\in (1,\infty)$ with $\theta:=\theta_0(1-\phi)+\theta_1\phi$,
\begin{align}
\label{eq:Dd_A_reiteration}
(\Dd(A^{\theta_0}),\Dd(A^{\theta_1}))_{\phi,q}
=(\Dd_A({\theta_0},q_0),\Dd_A({\theta_1},q_1))_{\phi,q}
=\Dd_A(\theta,q).
\end{align}

\section{Function spaces on the half line}\label{sec:Halfline}

In this section we discuss some results for function spaces on $\R$ and $\R_+$, which will be needed in the later sections.

\subsection{Besov, Triebel-Lizorkin, and Bessel potential spaces on intervals}
In this subsection we gather basic definition and facts of function spaces on some interval of $\R$. In this subsection, $X$ is a Banach space and $I=(0,T)$ for some $T\in (0,\infty]$. The following quotient definition is standard and can be found in \cite{Amannbook2,Rychkov1999,Saw18,Tr78,Tri83}.

\begin{definition}
\label{def:function_spaces_intervals}
Let $s\in \R$, $\g>-1$, $p\in (1,\infty)$ and $q\in [1,\infty]$. Let $\A^{s,p}\in \{H^{s,p},F^{s}_{p,q},B^{s}_{p,q}\}$. We denote by
$
\A^{s,p}(I,w_{\g};X)
$
the set of all $f\in \D'(I;X)$ for which there exists $g\in \A^{s,p}(\R,w_{\g};X)$ such that $g|_{I}=f$. Moreover, we set
\begin{equation*}
\|f\|_{\A^{s,p}(I,w_{\g};X)}:=
\inf\{\|g\|_{\A^{s,p}(\R,w_{\g};X)}\,:\,g|_{I}=f\}.
\end{equation*}
\end{definition}
With the help of extension operators one can find equivalent descriptions of these spaces. We will only need the following special extension operator.

\begin{proposition}\label{prop:ext}
Let $m\in \N_0$. Then there exists $(b_j)_{j=1}^{m+1}$ in $\R$ and $(\lambda_j)_{j=1}^{m+1}$ in $(0,\infty)$ such that for every Banach space $X$, $p\in (1, \infty)$, $q\in [1, \infty)$, $\gamma\in (-1,p-1)$, $|s|<m$ the mapping
$\Ext^m:B^{s}_{p,q}(\R_+,w_{\gamma};X)\to B^{s}_{p,q}(\R,w_{\gamma};X)$ given by
\begin{equation}
\label{eq:extension_op}
(\Ext^m f)(t):=
\begin{cases}
f(t),\qquad &t>0,\\
\sum_{j=1}^{m+1} b_j f(-\lambda_j t),\qquad &t<0;
\end{cases}
\end{equation}
is bounded, and $\Ext^m f|_{\R_+} = f|_{\R_+}$.

Moreover,
if $f\in L^p(\R_+,w_{\gamma};X)\cap C^{m}(\overline{\R_+};X)$, then $\Ext^m f\in C^{m}(\R;X)$.
\end{proposition}
\begin{proof}
Choosing the $(\lambda_j)_{j=1,\dots,m+1},(b_j)_{j=1,\dots,m+1}$ as in \cite[Subsection 2.9.3]{Tr78}, the result follows analogously to \cite[Proposition 5.6]{LMV18} where a suitable extension operator on $W^{m,p}(\R^d_+,w_{\gamma};X)$ is constructed. Note that unlike stated therein the weights in the latter reference should satisfy an additional scaling property. Now the claimed properties can be checked from the formulas in \cite[Proposition 5.6]{LMV18}. Here for $s\in (-m,0]$ one additionally needs to use the duality result (which does not require any conditions on $X$!)
\[B^{s}_{p,q}(\R,w_{\gamma};X)^*=B^{-s}_{p',q'}(\R,w_{\gamma'};X^*),\]
where $1/p+1/p' = 1$, $1/q+1/q'=1$ and $\gamma' = -\gamma/(p-1)$ (see \cite[Example 6.4]{Lindinter}), and the fact $\Sc(\R;X)$ is dense in $B^{s}_{p,q}(\R,w_{\gamma};X)$.
\end{proof}

\begin{remark}\label{rmk:prop:ext}\
\begin{enumerate}[{\rm(1)}]
    \item Extension operators of the type \eqref{eq:extension_op} are treated in \cite[Section~4.5.2]{Tri92II} for $B^{s}_{p,q}$ and $F^{s}_{p,q}$ with $p \in (0,\infty)$ and $q \in (0,\infty]$. There, the main ingredient is a characterization by means of oscillations, also see \cite{Lindinter}.
    One could extend this to the weighted Banach space-valued setting, where the case $p \in (1,\infty)$ and $w \in A_p$ would simplify.
    \item For the construction of a universal extension operator for $B^{s}_{p,q}$ and $F^{s}_{p,q}$ with $p \in (0,\infty)$ and $q \in (0,\infty]$ on Lipschitz domains we refer the reader to \cite{Rychkov1999}. One could extend this to the weighted Banach space-valued setting. Furthermore, let us note that \cite{Rychkov1999} also contains a preliminary extension result, \cite[Theorem~2.2]{Rychkov1999}.
\end{enumerate}
\end{remark}

\subsubsection*{Characterizations by local means on half-lines}

Let $g$ be a measurable function on $\R$. We say that $g$ satisfies the \emph{moment condition} of order $N \in \N_0 \cup \{-1\}$ if $\int_{\R} x^k g(x) dx = 0$ for all $k \leq N$.

Let $I=(a,\infty)$ with $a \in \R$.
For $f \in \mathscr{D}'(I;X)$ and $\chi \in \mathscr{D}(\R_-)$ we set $f*\chi(x):= f(\chi(x-\,\cdot\,))$ for each $x \in I$, which is well-defined because $\chi(x-\,\cdot\,) \in \mathscr{D}(I)$ for every such $x$. Then $f*\chi$ is an $X$-valued $C^\infty$-function on $I$.

Furthermore, we define $\TD(I;X)$ as the subspace of $\D'(I;X)$ consisting of all $f \in \D'(I;X)$ having finite order and at most polynomial growth at infinity, that is, for which there exist $C \in [0,\infty)$ and $N \in \N_0$ such that
$$
\|f(\phi)\|_X \leq C\sup_{k \leq N}\sup_{t \in I}(1+|t|)^N|\partial^{k}_t\phi(t)|, \qquad \phi \in \D(I).
$$

The following theorem is an extension of  \cite[Theorem~3.2]{Rychkov1999} to the weighted Banach space-valued setting. It will be needed in one of the estimates in the difference characterization in Proposition \ref{prop:equivalentNormF} below.
\begin{theorem}\label{thm:local_mean_Rychkov}
Let $I=(a,\infty)$ with $a \in \R$, $p \in [1,\infty)$, $q \in [1,\infty]$, $w \in A_\infty(\R)$ and $s \in \R$.
Let $\phi_0,\phi \in \mathscr{D}(\R_-)$ satisfy $\phi(x)=\phi_0(x)-\frac{1}{2}\phi(\frac{x}{2})$. Set $\phi_j(x)=2^j\phi(2^j x)$ for each $j \geq 1$.
Suppose that $\int_{\R}\phi_0(x) dx \neq 0$ and that $\phi$ satisfies the moment condition of order $N \geq \lfloor s \rfloor$. Then
\begin{equation*}
\|f\|_{F^s_{p,q}(I,w;X)} \eqsim_{p,q,w,s}
\|(2^{js}\phi_j*f)_{j\geq 0}\|_{L^p(I,w;\ell^q(X))}, \ \ \ f \in \TD(I;X).
\end{equation*}
\end{theorem}
A variant of Theorem \ref{thm:local_mean_Rychkov} hold with $F$ replaced by $B$ if $L^p(I,w;\ell^q(X))$ is replaced by $\ell^q(L^p(I,w;X))$.

\subsection{Characterization by differences}
For $a\in \R$, let $I \in \{(a,\infty),\R\}$ and $m \in \N$.
We define
\begin{align*}
V^m_I(x,t) &:= \{ h \in \R : |h|<t, x+mh \in I \}
= \left\{\begin{array}{lr}
(-t,t),     &  I=\R,\\
(-t,t) \cap (\frac{a}{m}-\frac{x}{m},\infty),     &  I=\R_+,
\end{array}\right.
\end{align*}
for each $x \in I$ and $t > 0$.
For a strongly measurable function $f:I \to X$ we define the $m$-th order difference of $f$ as
$$
\Delta_h^m f(x):=\sum_{j=0}^m \binom{m}{j} (-1)^j f(x+(m-j)h),\qquad x \in I, h \in \R, x+m \in I,
$$
and we define the means of the $m$-th order difference of $f$ as
$$
d^m_t f(x) := t^{-1}\int_{V^m_I(x,t)}\|\Delta_h^m f(x)\|_X dh, \qquad x \in I, t>0.
$$

Let $s>0$, $p \in (1,\infty)$ and $w \in A_p(\R)$.
For $f\in L^p(I,w;X)$ we set
\begin{align}\label{eq:defFmnorm}
[f]_{F^s_{p,q}(I,w;X)}^{(m)}
&:=\Big\| \Big( \int_0^{\infty} \Big( t^{-s} d^m_t f(\,\cdot\,) \Big)^q \frac{dt}{t} \Big)^{1/q}\Big\|_{L^p(I,w)},
\end{align}
with the usual modification for $q=\infty$, and
$$
\norm{f}_{F^s_{p,q}(I,w;X)}^{(m)}:=\|f\|_{L^p(I,w;X)}+[f]_{F^s_{p,q}(I,w;X)}^{(m)}.
$$
By a discretization argument, for any $\kappa>0$,
\begin{align}\label{eq:Fmdiscr}
[f]_{F^s_{p,q}(I,w;X)}^{(m)} \eqsim_{s,\kappa} \|(2^{js}d^m_{2^{-j}\kappa}f)_{j \in \Z}\|_{L^p(I,w;\ell^q(X))}.
\end{align}

One can obtain the following extension to the weighted setting of a well-known result on the characterization by means of differences for $F$-spaces  (see \cite[Proposition~6]{SS04}, \cite[Section 2.5.11]{Tri83}, \cite[Sections 3.5.3 and 4.5.4]{Tri92II} and \cite[Theorem 6.9]{Tri06III}).

\begin{proposition}
\label{prop:equivalentNormF}
Let $p\in (1,\infty)$, $q\in [1,\infty]$, $w\in A_p(\R)$ and $s>0$. Let $m$ be an integer such that $m>s$. Let $I\in \{\R,(a,\infty)\}$ where $a\in \R$ is fixed.
Then there is the equivalence of (extended) norms
\begin{equation}
\label{eq:equivalence_norm}
\|f\|_{F^s_{p,q}(I,w;X)} \eqsim \norm{f}_{F^s_{p,q}(I,w;X)}^{(m)}, \ \ \ f\in L^p(I,w;X).
\end{equation}
\end{proposition}
\begin{proof}
The case $I = \R$ follows as in the above references.

Next consider $I = (a, \infty)$. By a translation argument we may assume $a=0$. The inequality '$\gtrsim$' in \eqref{eq:equivalence_norm} follows from the corresponding inequality for $I = \R$.
Indeed, given $f \in F^s_{p,q}(\R_+,w;X)$ and $g \in F^s_{p,q}(\R,w;X)$ with $f=g|_{\R_+}$, we have
\begin{equation*}
\norm{f}_{F^s_{p,q}(\R_+,w;X)}^{(m)} \leq \norm{g}_{F^s_{p,q}(\R,w;X)}^{(m)} \lesssim \|g\|_{F^s_{p,q}(\R,w;X)}
\end{equation*}
and '$\gtrsim$' in \eqref{eq:equivalence_norm} follows by taking the infimum over all such $g$.

In order to prove the reverse inequality '$\lesssim$' in \eqref{eq:equivalence_norm}, fix $f \in L^p(\R_+,w;X)$.
Let $k_0,k \in \mathscr{D}(\R_+)$ be as in \cite[Section~3.3.2]{Tri92II}.
Then $\phi_0 := k_0(-\,\cdot\,)$ and $\phi := k(-\,\cdot\,)$ satisfy the conditions of Theorem~\ref{thm:local_mean_Rychkov}, and there exists $\chi \in \mathscr{D}(\R_+)$ such that
$$
\phi_j*f(x) = \int_{\R_+} \partial_y^m \chi(y) \Delta^m_{2^{-j}y}f(x)dy, \quad x \in \R_+, j \geq 1,
$$
implying that
\begin{align*}
\|\phi_j*f(x)\|_X
&\lesssim \int_{\supp(\chi)} \|\Delta^m_{2^{-j}y}f(x)\|_X dy\lesssim d^m_{2^{-j}\kappa}f(x), \quad x \in \R_+, j \geq 1,
\end{align*}
for $\kappa > 0$ such that $\supp(\chi) \subseteq (0,\kappa)$.
As $p \in (1,\infty)$ and $w \in A_p(\R)$, we have $\|\psi_0*f\|_{L^p(\R_+,w;X)} \lesssim \|f\|_{L^p(\R_+,w;X)}$.
It thus follows from Theorem~\ref{thm:local_mean_Rychkov} that
\begin{align*}
\|f\|_{F^s_{p,q}(\R_+,w;X)}
&\lesssim \|\psi_0*f\|_{L^p(\R_+,w;X)}+\|(2^{js}\phi_j*f)_{j \geq 1}\|_{L^p(\R_+,w;\ell^q(X))} \\
&\lesssim \|f\|_{L^p(\R_+,w;X)}+\|(2^{js}d^m_{2^{-j}\kappa}f)_{j \geq 1}\|_{L^p(\R_+,w;\ell^q(X))}
\lesssim [f]_{F^s_{p,q}(\R_+,w;X)}^{(m)}
\end{align*}
where in the last step we used \eqref{eq:Fmdiscr}.
\end{proof}

\begin{remark}\label{rem:seminorms}
Let $s>0$, $m,n\in \N$ with $m,n>s$. If $\g\in (-1,p-1)$, then setting $f_{\lambda}=f(\lambda \cdot)$ in \eqref{eq:equivalence_norm} and letting $\lambda\to \infty$ one obtains
$$
[f]_{F^{s}_{q,p}(I,w_{\g};X)}^{(m)}\eqsim [f]_{F^{s}_{q,p}(I,w_{\g};X)}^{(n)},
$$
for all $f \in F^{s}_{q,p}(I,w_{\g};X)$.
Therefore, for the weight $w_{\g}$, throughout the paper we write $[f]_{F^{s}_{q,p}(I,w_{\g};X)}$ instead of $[f]_{F^{s}_{q,p}(I,w_{\g};X)}^{(m)}$ whenever this is convenient.
\end{remark}

\subsection{Connections with fractional powers}\label{sec:characfractional}
Let $p\in (1, \infty)$ and $\gamma\in (-1,p-1)$. On $L^p(\R_+,w_{\gamma};X)$ consider the following derivative operators:
\begin{align}\label{eq:defC}
\Der u &= u' & \text{with $\Do(\Der) = \Wz^{1,p}(\R_+,w_{\gamma};X)$,}
\\
\label{eq:defB}
\Dert u &= -u' & \text{with $\Do(\Dert) = W^{1,p}(\R_+,w_{\gamma};X)$,}
\end{align}
where $\Wz^{1,p}(\R_+,w_{\gamma};X):=\{u\in W^{1,p}(\R_+,w_{\gamma};X)\,:\, u(0)=0\}$. By standard considerations, one can check that the trace $u(0)$ for $u\in W^{1,p}(\R_+,w_{\gamma};X)$ is well-defined  (see Lemma \ref{lemma:t:trace_FF_FL_spaces;X_0} and \eqref{eq:EmbElementary}).

Then for $\gamma\in [0,p-1)$, $(e^{-t \Dert})_{t\geq 0}$ is the contractive $C_0$-semigroup given by left-translation.
Fix integers $0<n<m$. By the Balakrishnan formula for the fractional power (see \cite[Theorem 3.2.2]{MartinezSanz}) we have that for $s\in (0,n)$ and $u\in \Do(\Dert^n)$
\begin{align}\label{eq:Bpowers}
\Dert^{s} u= C_{s,m} \int_0^\infty \frac{(\Id-e^{-t\Dert})^m u}{t^{1+s}} dt.
\end{align}
Moreover, when $u\in L^p(\R_+,w_{\gamma};X)$, and
$\lim_{\delta\downarrow 0}\int_\delta^\infty \frac{(\Id-e^{-t\Dert})^m u}{t^{1+s}} dt$ exists in $X$,
$u\in \Do(\Dert^{s})$ and \eqref{eq:Bpowers} holds.
In particular, this is the case if
\begin{align}\label{eq:fractionalpowerB}
\Big\| x \mapsto \int_0^\infty t^{-1-s}\|(\Id - e^{-t\Dert} )^m u(x) \|_{X}dt \Big\|_{L^p(\R_+,w_{\gamma})}<\infty.
\end{align}
Let us compare this to the difference norm of $F^s_{p,q}$ introduced in \eqref{eq:defFmnorm}.  By Fubini's theorem
\[[u]_{F^s_{p,1}(I,w_{\gamma};X)}^{(m)} = c_{s}
\Big\| x \mapsto \int_{-mx}^\infty |h|^{-1-s} \|\Delta_h^m u(x) \|_X dh \Big\|_{L^p(\R_+,w_{\gamma})}.\]
where $c_{s}>0$ depends only on $s$. Therefore, by \eqref{eq:fractionalpowerB} it is immediate that $u\in \Do(\Dert^s)$ and
$\|\Dert^s u \| \lesssim [f]_{F^s_{p,1}(\R_+,w_{\gamma};X)}^{(m)}$. This implies $F^{s}_{p,1}(\R_+,w_{\gamma};X) \hookrightarrow \Do(\Dert^s)$.

Next we show that one can obtain a Sandwich lemma of similar nature as \eqref{eq:EmbElementary}. If one assumes $X$ has UMD (see e.g.\ \cite[Chapter 4]{Analysis1}), then the embedding in the next result actually follows from \eqref{eq:EmbElementary} and \cite[Theorem 6.8(2)]{LMV18}.
\begin{lemma}[Sandwich lemma]\label{lem:sandwichHRplus}
Let $s>0$ and let $m>s$ be an integer. Let $p\in (1, \infty)$ and $\gamma\in [0,p-1)$. Then $F^{s}_{p,1}(\R_+,w_{\gamma};X) \hookrightarrow \Do(\Dert^s)\hookrightarrow F^{s}_{p,\infty}(\R_+,w_{\gamma};X)$, Moreover, for all $u\in F^{s}_{p,1}(\R_+,w_{\gamma};X)$,
\begin{align}\label{eq:equivhom}
[u]_{F^{s}_{p,\infty}(\R_+,w_{\gamma};X)} \lesssim \|\Dert^{s} u\|_{L^p(\R_+,w_{\gamma};X)}\lesssim [u]_{F^{s}_{p,1}(\R_+,w_{\gamma};X)},
\end{align}
\end{lemma}
\begin{proof}
The second embedding follows from the above discussion, but one could also use a variation of the argument below. For the first one, we will first show that
\begin{equation}\label{eq:lem:sandwichHRplus}
\|u\|_{F^{s}_{p,\infty}(\R_+,w_{\gamma};X)}\leq C\|(1+\Dert)^s u\|_{L^p(\R_+,w_{\gamma};X)},
\end{equation}
or equivalently, that $\|(1+\Dert)^{-s} v\|_{F_{p,\infty}^s(\R_+,w_{\gamma};X)}\leq C \|v\|_{L^p(\R_+,w_{\gamma};X)}$ for all $v\in L^p(\R_+,w_{\gamma};X)$. Let $Ev$ denote the zero extension of $v$ to $\R$, let $R$ denote the restriction to $\R_+$, and let $\Dert_{\R} u= -u'\in L^p(\R,w_{\g};X)$ for $u\in \Do(\Dert_{\R})=W^{1,p}(\R,w_{\g};X)$. Then from \cite[the proof of Theorem 6.8(2)]{LMV18}, $(1+\Dert)^{-s}v = R (1+\Dert_{\R})^{-s} E v$ holds for $s=1$, and therefore, by definition of the fractional power the same holds for $s\in (0,1)$. This implies
\begin{align*}
\|(1+\Dert)^{-s} v\|_{F_{p,\infty}^s(\R_+,w_{\gamma};X)} &\leq \|(1+\Dert_{\R})^{-s} E v\|_{F_{p,\infty}^s(\R,w_{\gamma};X)}
\\ & \stackrel{(i)}{\eqsim} \|J_s(1+\Dert_{\R})^{-s} E v\|_{F_{p,\infty}^0(\R,w_{\gamma};X)}
\\ &  \stackrel{(ii)}{\lesssim} \|E v\|_{F_{p,\infty}^0(\R,w_{\gamma};X)}
\\ & \lesssim \|E v\|_{L^p(\R,w_{\gamma};X)}\leq \|v\|_{L^p(\R_+,w_{\gamma};X)},
\end{align*}
where $J_s:=\Four^{-1}((1+|\cdot|^2)^{s/2}\Four(\cdot))$. Here $(i)$ follows from \cite[Proposition 3.9]{MV12} and $(ii)$ follows from Fourier multiplier theory in Triebel--Lizorkin spaces (see \cite[15.6]{Tr97} and the comments in \cite[Proposition 3.10]{MV12}).

It remains to prove the first inequality in \eqref{eq:equivhom}. For this we use the shorthand notation $L^p$ and $F^{s}_{p,q}$. By \eqref{eq:lem:sandwichHRplus} and \cite[Proposition 3.1.9]{Haase:2} we have
\[[u]_{F^{s}_{p,\infty}} + \|u\|_{L^p} \lesssim \|\Dert^{s} u\|_{L^p}+ \|u\|_{L^p}.
\]
Applying this to $u_{\lambda}(x) = u(\lambda x)$ and using the scaling properties implied by the difference norms and the expression \eqref{eq:Bpowers}, after rewriting we obtain
\[[u]_{F^{s}_{p,\infty}} + \lambda^{-s}\|u\|_{L^p} \lesssim \|\Dert^{s} u\|_{L^p}+ \lambda^{-s}\|u\|_{L^p}.
\]
Letting $\lambda\to \infty$ this gives the first inequality in \eqref{eq:equivhom}.
\end{proof}

\begin{remark}\label{rem:generalweights}
One can actually show that $\Dert$ is sectorial on $L^p(\R_+,w_{\gamma};X)$ for any $\gamma\in (-1,p-1)$. In this case $-\Dert$ no longer generates a $C_0$-semigroup for $\gamma\in (-1,0)$. However, the first part of Lemma \ref{lem:sandwichHRplus} on the embedding result $\Do(\Dert^s)\hookrightarrow F^{s}_{p,\infty}(\R_+,w_{\gamma};X)$ extends to this setting if one uses the Fourier multiplier argument of that proof. For this we only need to check that $\Dert$ is sectorial. For this consider $\Dert_{\R} u =- u'\in L^p(\R,w;X)$ for $u\in W^{1,p}(\R,w;X)$. This is sectorial for $w=1$. One can check that $\lambda(\lambda+\Dert_{\R})^{-1}$ is an integral operator with kernel $k_{\lambda}(t) = \lambda e^{t\lambda} \one_{(-\infty,0)}$ for $\lambda>0$. Now it remains to extrapolate in the weight $w_{\g}$ with $\g\in (-1,p-1)$ as in \cite{Kree,stein1957note} (also see \cite[Proposition 2.3]{PruSim04}).
Therefore, $\lambda(\lambda+\Dert_{\R})^{-1}$ is uniformly bounded on $L^p(\R,w_{\g};X)$ for $\lambda>0$ and hence $\Dert_{\R}$ is sectorial. Arguing as in \cite[Theorem 6.8]{LMV18} one sees that $\Dert$ is also sectorial. The above for instance implies that the first part of  Corollary \ref{c:t:trace_FF_FL_spaces;H-hom} can be extended to $\gamma_i\in (-1,0)$ (see the comments before the statement of Corollary \ref{c:t:trace_FF_FL_spaces;H-hom}).

Using the more advanced extrapolation theory from \cite[Corollary 2.10]{HanHyt} one can even obtain the above results for $w\in A_p$.
\end{remark}

For $s \in \R$ we set $\Hz^{s,p}(\R_+,w_{\g};X):=
\overline{C^{\infty}_{c}(\R_+;X)}^{H^{s,p}(\R_+,w_{\g};X)}$.
If $k<s - \tfrac{1+\gamma}{p}<k+1$ with $k\in \N_0\cup \{-1\}$, then (see \cite[Proposotion 6.4]{LMV18})
\[\Hz^{s,p}(\R_+,w_{\gamma};X) = \{u\in H^{s,p}(\R_+,w_{\gamma};X): \Tr^{k} u = 0 \},\]
where we define $\Tr^{-1} u=0$. In the next result we compare the fractional domain spaces of $\Dert$ and $\Der$ under the assumption that $X$ is a UMD space (see \cite{Analysis1} for details).
\begin{proposition}\label{prop:Hinftyder}
Let $X$ be a UMD space, $p\in (1, \infty)$, $\gamma\in (-1, p-1)$ and $s>0$. Then the operators $\Dert$ and $\Der$ have a bounded $H^\infty$-calculus of angle $\pi/2$, and
\begin{align*}
\Do(\Der^s) \hookrightarrow \Do(\Dert^s) & = H^{s,p}(\R_+,w_{\gamma};X).
\end{align*}
Moreover, if $s\notin \N_0+(1+\gamma)/p$, then
\begin{equation}\label{eq:fractionalBC}
\Do(\Der^s) = \Hz^{s,p}(\R_+,w_{\gamma};X), \ \text{and} \  \|\Dert^s u\|_{L^p(\R_+,w_{\gamma};X)} \eqsim \|\Der^s u\|_{L^p(\R_+,w_{\gamma};X)}.
\end{equation}
for all $u\in \Do(\Der^s)$.
\end{proposition}
\begin{proof}
The identities for the fractional domain spaces are immediate from \cite[Proposition 5.6 and Theorem 6.8]{LMV18}. The latter paper also gives that
$\Der$ and $\Dert$ have a bounded $H^\infty$-calculus of angle $\pi/2$, and in particular bounded imaginary powers.
To prove the embedding $\Do(\Der^s) \hookrightarrow \Do(\Dert^s)$ we use the previous identities with an integer $n>s$, and \cite[Theorem 6.6.9]{Haase:2} to find
\begin{align*}
\Do(\Der^s) & = [L^p(\R_+, w_{\gamma};X),\Do(\Der^n)]_{\frac{s}{n}}
\\ & = [L^p(\R_+, w_{\gamma};X),\Hz^{n,p}(\R_+,w_{\gamma};X)]_{\frac{s}{n}} \\ & \hookrightarrow [L^p(\R_+,w_{\gamma};X),H^{n,p}(\R_+,w_{\gamma};X)]_{\frac{s}{n}}
\\ & =[L^p(\R_+, w_{\gamma};X),\Do(\Dert^n)]_{\frac{s}{n}} = \Do(\Dert^s).
\end{align*}
For the final part of \eqref{eq:fractionalBC}, it suffices to consider $s\in (0,1)$, and by density $u\in \Do(\Der)$. Note that the first part implies that \[\|\Dert^s u\|_{L^p(\R_+,w_{\gamma};X)} + \|u\|_{L^p(\R_+,w_{\gamma};X)} \eqsim \|\Der^s u\|_{L^p(\R_+,w_{\gamma};X)} + \|u\|_{L^p(\R_+,w_{\gamma};X)}.\]
Applying the latter to $u_{\lambda} = u(\lambda x)$, and using \cite[Proposition 3.1.12]{Haase:2}
one obtains
\[\|\Dert^s u\|_{L^p(\R_+,w_{\gamma};X)} + \lambda^{-s}\|u\|_{L^p(\R_+,w_{\gamma};X)} \eqsim \|\Der^s u\|_{L^p(\R_+,w_{\gamma};X)} + \lambda^{-s}\|u\|_{L^p(\R_+,w_{\gamma};X)}.\]
Letting $\lambda\to \infty$, the desired estimate follows.
\end{proof}
It would be interesting to know whether \eqref{eq:fractionalBC} holds without the  condition that $X$ is a UMD space.

Finally, we collect some standard properties of $\Der$ and $\Der_T$. Here for $T\in (0,\infty]$, $p\in (1, \infty)$, and $\gamma\in (-1,p-1)$ we set $\Der_T u = u'$  with $\Do(\Der_T) = \Wz^{1,p}(0,T,w_{\gamma};X)$, where the boundary condition only applies to the left endpoint (see below \eqref{eq:defB}). By setting $\Der_{\infty} = \Der$, we will treat both cases at the same time. We claim that for all $\lambda \in \C_+$ with $\Re(\lambda)>0$,
\begin{equation}\label{eq:resolventD}
 (\lambda + \Der_T)^{-1} f(t) = \int_0^t e^{-\lambda(t-s)}f(s) ds, \ \  f\in L^p(0,T,w_{\gamma};X).
\end{equation}
It is standard to check \eqref{eq:resolventD} for $f\in C^\infty_c([0,T];X)$.
Therefore, it remains to check the boundedness of the operator on the RHS of \eqref{eq:resolventD}. The uniform boundedness of $\Re(\lambda)(\lambda + \Der_T)^{-1}$  for $\gamma=0$ is clear from Young's inequality. The boundedness for $\gamma\in (-1, p-1)$ can be deduced from the case $\gamma=0$ by the same extrapolation argument as referred to in Remark \ref{rem:generalweights}. Moreover, if $T<\infty$, one can additionally check that $\Der_T$ is invertible and the inverse is given by the RHS of \eqref{eq:resolventD} for $\lambda=0$, which can be checked via Hardy--Young's inequality (see \eqref{eq:HardyYoung} below).

The above implies that $\Der_T$ is a sectorial operator and since
\begin{align}\label{eq:idenresolveTinfty}
(\lambda+\Der_T)^{-1}f = ((\lambda+\Der)^{-1} E_T f)|_{(0,T)},
\end{align}
where $E_T$ denotes the extension by zero on $(0,\infty)$, it also follows that $\Der_T$ inherits the bounded $H^\infty$-calculus of angle $\pi/2$ of $\Der$ (see Proposition \ref{prop:Hinftyder}) with uniform estimates in $T$. Moreover, by \eqref{eq:idenresolveTinfty} and the extended calculus (see \cite[Section 3.1]{Haase:2})
\begin{align}\label{eq:DerTDer}
\Der_T^{\alpha} f = \Der^{\alpha} f|_{(0,T)}, \qquad \alpha\in \R, \ f\in \Do(\Der^{\alpha}),
\end{align}
and this extends to $f\in \Dd(\Der^{\alpha})$ by density and closedness.

Next we claim that
\begin{align}\label{eq:DerTinverseid}
 \Der_T^{-\alpha} f = \mathcal{K}_{\alpha}* E_T f, \qquad \alpha>0, \ f\in \Ran(\Der_T^{ \alpha}),
\end{align}
where $\mathcal{K}_{\alpha}(t):=\frac{|t|^{\alpha-1} \one_{(0,\infty)}(t)}{\Gamma(\alpha)}$. Indeed, the case $\alpha \in (0,1)$ follows from \cite[Corollary 3.1.14 and Proposition 3.2.1]{Haase:2} and elementary calculations.
Now by analyticity in $\alpha$, \eqref{eq:DerTinverseid} extends to all $\alpha>0$. Since $0\in \rho(\Der_T)$ for all $T<\infty$, we have $\Ran(\Der_T^{ \alpha}) = L^p(0,T,w_{\gamma};X)$. In the case $T=\infty$ and $\alpha>0$ one has
\begin{equation}\label{eq:inverseKalpha}
\mathcal{K}_{\alpha}* f\in \Dd(\Der^{\alpha}) \ \ \text{and} \ \  \Der^{\alpha} \mathcal{K}_{\alpha}*f = f, \ \ \ f\in L^p(\R_+,w_{\gamma};X).
\end{equation}
To see this, by density we may assume $f\in \Ran(\Do(\Der^{\alpha}))$ and by \eqref{eq:DerTDer} and \eqref{eq:DerTinverseid} we get, for all $T<\infty$,
\[(\Der^{-\alpha} f)|_{(0,T)}=\Der_{T}^{-\alpha} f =\mathcal{K}_{\alpha}* E_T f= (\mathcal{K}_{\alpha} * f)|_{(0,T)} \ \ \text{on $(0,T)$.}\]
The arbitrariness of $T<\infty$ readily yields \eqref{eq:inverseKalpha}.

\subsection{Sobolev embedding with power weights}
The following embedding result is the half line version of one of the main results of \cite{MV12} by Meyries and the third named author.
\begin{theorem}[Sobolev embedding]
\label{t:sobolevTriebel}
Let $I \in \{\R_+,\R\}$, $X$ a Banach space, $1<p_0\leq p_1<\infty$, $q_0,q_1\in [1,\infty]$, $s_0>s_1$ and $\g_0,\g_1>-1$. Assume
\begin{equation}\label{eq:condSobembg}
\frac{\g_0}{p_0}\geq \frac{\g_1}{p_1},\qquad s_0-\frac{1+\g_0}{p_0}=s_1-\frac{1+\g_1}{p_1}.
\end{equation}
Then
\begin{equation}\label{eq:sob_embedding_Triebel}
F^{s_0}_{p_0,q_0}(I,w_{\g_0};X) \hookrightarrow   F^{s_1}_{p_1,q_1}(I,w_{\g_1};X).
\end{equation}
If in addition $s_0,s_1>0$ and $\g_i < p_i-1$ for $i\in \{0,1\}$ and we are given $\N\ni m_i>s_i$ for $i\in \{0,1\}$, then for all $f \in F^{s_0}_{p_0,q_0}(I,w_{\g_0};X)$ one has
\begin{equation}
\label{eq:sob_embedding_Triebel_seminorm}
[f]_{F^{s_1}_{p_1,q_1}(I,w_{\g_1};X)}^{(m_1)}
\lesssim [f]_{F^{s_0}_{p_0,q_0}(I,w_{\g_0};X)}^{(m_0)}.
\end{equation}
\end{theorem}

\begin{proof}
The embedding \eqref{eq:sob_embedding_Triebel} for $I=\R$ is contained in \cite{MV12}, from which the corresponding embedding for $I=\R_+$ follows.
To prove the last claim, fix $f \in F^{s_0}_{p_0,q_0}(I,w_{\g_0};X)$.
Let $\lambda>0$ and set $f_{\lambda}:=f(\lambda \cdot)$. By applying the norm inequality associated with the embedding \eqref{eq:sob_embedding_Triebel} to $f_{\lambda}$, we obtain
\begin{align*}
[f_{\lambda}]_{F^{s_1}_{p_1,q_1}(I,w_{\g_1};X)}^{(m_0)}
&\lesssim \|f_{\lambda}\|_{F^{s_1}_{p_1,q_1}(I,w_{\g_1};X)}\\
 &\lesssim \|f_{\lambda}\|_{F^{s_0}_{p_0,q_0}(I,w_{\g_0};X)}\\
 &= \|f_{\lambda}\|_{L^{p_0}(I,w_{\g};X)}+ [f_{\lambda}]_{F^{s_0}_{p_0,q_0}(I,w_{\g_0};X)}^{(m_1)}.
\end{align*}
By standard computations, $[f_{\lambda}]_{F^{s_i}_{p_i,q_i}(I,w_{\g_i};X)}^{(m_i)}=\lambda^{\delta_i}[f]_{F^{s_i}_{p_i,q_i}(I,w_{\g_i};X)}^{(m_i)}$ for $i\in \{0,1\}$ where $\delta_i:=s_i-\frac{1+\g_i}{p_i}$. Then, by assumption we have $\delta_0=\delta_1$, and we find that
\begin{align*}
[f]_{F^{s_1}_{p_1,q_1}(I,w_{\g_1};X)}^{(m_1)}
\lesssim \lambda^{-s_0}\|f\|_{L^{p_0}(I,w_{\g_0};X)}
+[f]_{F^{s_0}_{p_0,q_0}(I,w_{\g_0};X)}^{(m_0)}.
\end{align*}
The claim follows by letting $\lambda\to \infty$.
\end{proof}

\begin{corollary}[Sobolev embedding]
\label{c:t:sobolevTriebel}
Let $1<p_0\leq p_1<\infty$, $q_0 \in [1,\infty]$, $s_0>0$, $-1<\g_0<p_0-1$, $-1<\g_0<p_1-1$ and $\N\ni m_0>s_0$. Assume \eqref{eq:condSobembg}.
Then
\begin{align}\label{eq:c:sob_embedding_Triebel}
F^{s_0}_{p_0,q_0}(I,w_{\g_0};X)& \hookrightarrow   L^{p_1}(I,w_{\g_1};X).
\end{align}
and for all $f \in F^{s_0}_{p_0,q_0}(I,w_{\g_0};X)$,
\begin{equation}
\label{eq:c:sob_embedding_Triebel_seminorm}
\|f\|_{L^{p_1}(I,w_{\g_1};X)}
\lesssim [f]_{F^{s_0}_{p_0,q_0}(I,w_{\g_0};X)}^{(m_0)}.
\end{equation}
\end{corollary}
\begin{proof}
The embedding \eqref{eq:c:sob_embedding_Triebel} follows from the Sobolev embedding \eqref{eq:sob_embedding_Triebel} with $s_1=0$ and $q_1=1$ and \eqref{eq:EmbElementary}.
The inequality \eqref{eq:c:sob_embedding_Triebel_seminorm} can subsequently be derived by a scaling argument, as in the proof of Theorem~\ref{t:sobolevTriebel}.
\end{proof}

We conclude this subsection with a useful interpolation result from \cite[Theorem 3.1]{MV14} by Meyries and the third named author, where there is a misprint in the definition of $w_{\theta}$.

\begin{theorem}[Mixed-derivatives]
\label{t:interpolationA}
Let $I \in \{\R_+,\R\}$ and let $X_0,X_1$ be an interpolation couple of Banach spaces. Let $\theta\in (0,1)$ and $X_{\theta}$ be a Banach space such that $X_0\cap X_1 \subseteq X_{\theta}\subseteq X_0+X_1$ and
\begin{equation}
\label{eq:interpolation_inequality}
\|x\|_{X_{\theta}} \leq C  \|x\|_{X_0}^{1-\theta}\|x\|_{X_1}^{\theta}, \quad x\in X_{0}\cap X_1.
\end{equation}
Assume $p_0,p_1,p\in (1,\infty)$, $q_0,q_1,q\in [1,\infty]$, $w,w_0, w_1\in A_{\infty}$ satisfy
$$
\frac{1}{p}=\frac{1-\theta}{p_0}+\frac{\theta}{p_1},
\quad \frac{1}{q}=\frac{1-\theta}{q_0}+\frac{\theta}{q_1},
\quad\text{and}\quad
w^{1/p}_{\theta} = w_0^{(1-\theta)/p_0} w_1^{\theta/p_1}.
$$
Let $-\infty<s_1<s_0<\infty$ and set $s:=s_1+(1-\theta)(s_0-s_1)$.
Then for $\mathcal{A}\in \{B, F\}$,
$$
\mathcal{A}^{s_0}_{p_0,q_0}(I,w_{0};X_0)\cap \mathcal{A}^{s_1}_{p_1,q_1}(I,w_{1};X_1)\hookrightarrow
\mathcal{A}_{p,q}^{s}(I,w_\theta;X_{\theta}),
$$
and for all $f\in \mathcal{A}^{s_0}_{p_0,q_0}(I,w_{0};X_0)\cap \mathcal{A}^{s_1}_{p_1,q_1}(I,w_{1};X_1)$ one has
\begin{equation}
\label{eq:mixed_derivative_norm}
\|f\|_{\mathcal{A}_{p,q}^{s}(I,w_{{\theta}};X_{\theta})} \lesssim
 \|f\|_{\mathcal{A}^{s_0}_{p_0,q_0}(I,w_{0};X_0)}^{1-\theta}\|f\|_{\mathcal{A}^{s_1}_{p_1,q_1}(I,w_{1};X_1)}^{\theta}.
\end{equation}
Moreover, if $s_i>0$, $w_i(t)=w_{\g_i}(t) = |t|^{\gamma_i}$ with $\g_i\in (-1,p_i-1)$ for $i\in \{0,1\}$, then
\begin{equation}
\label{eq:mixed_derivative_seminorm}
[f]_{\mathcal{A}_{p,q}^{s}(I,w_{\g_{\theta}};X_{\theta})} \lesssim
[f]_{\mathcal{A}^{s_0}_{p_0,q_0}(I,w_{\g_0};X_0)}^{1-\theta}
[f]_{\mathcal{A}^{s_1}_{p_1,q_1}(I,w_{\g_1};X_1)}^{\theta},
\end{equation}
where $\frac{\g_{\theta}}{p}=(1-\theta)\frac{\g_0}{p_0}+\theta\frac{\g_1}{p_1}$.
\end{theorem}
\begin{proof}
The inequality \eqref{eq:mixed_derivative_norm} for $I=\R$ follows from the definitions by applying the H\"{o}lder inequality twice, see \cite[Theorem 3.1]{MV14}). The case $I=\R_+$ can in the same way be obtained from the intrinsic characterization in Theorem~\ref{thm:local_mean_Rychkov}.
The homogeneous estimate \eqref{eq:mixed_derivative_seminorm} subsequently follows
by Proposition~\ref{prop:equivalentNormF} with a scaling argument as in the proof of Theorem \ref{t:sobolevTriebel}.
\end{proof}

In the case $q_0=p_0$, $q_1=p_1$, $w_0=w_{\g_0}$ with $\g_0 \in (-1,p_0-1)$ and $w_1=w_{\g_1}$ with $\g_1 \in (-1,p_1-1)$, the result for $I=\R_+$ can also be derived by an extension argument from Proposition~\ref{prop:ext}. In this way the more advanced result of Theorem~\ref{thm:local_mean_Rychkov}, can be avoided.

\section{Trace Embeddings for Triebel-Lizorkin spaces}
\label{s:trace}
In this section we study trace embedding results for anisotropic function spaces.

\subsection{Traces at the origin}
We start with a lemma on the well-definedness of higher order trace operators in the isotropic setting.

\begin{lemma}\label{lemma:t:trace_FF_FL_spaces;X_0}
Let  $I \in \{\R_+,\R\}$, $X$ a Banach space, $p \in (1,\infty)$, $q \in [1,\infty]$, $\g \in (-1,\infty)$, $s \in \R$ and $k \in \N_0$.
Let $s > k+\frac{1+\g}{p}$. Then
the $k$-th order trace operator $\Tr^k:F^s_{p,q}(I,w_\g;X) \cap C^k(\bar{I};X) \to X$ given by $\Tr^k u=\partial^k_t u\,(0)$ has a unique extension to a bounded linear operator $\Tr^k:F^s_{p,q}(I,w_\g;X) \to X$ that is also bounded from $\big(F^s_{p,q}(I,w_\g;X),\|\,\cdot\,\|_{F^\sigma_{p,1}(I,w_\g;X)}\big)$ to $X$ for every $\sigma \in (k+\frac{1+\g}{p},s)$.

The same holds with $F^s_{p,q}(I,w_\g;X)$ replaced by $H^{s,p}(I,w_\g;X)$. \end{lemma}

The formulation of the lemma can be slightly confusing at first reading. The final part on $F^s_{p,q}(I,w_\g;X)$ with the artificial norm $\|\,\cdot\,\|_{F^\sigma_{p,1}(I,w_\g;X)}$ is in order to obtain uniqueness of the extension in the case $q=\infty$. Below we will use the simple trace estimate obtained in \cite[Proposition~6.3]{LMV18} and from the proof of the latter it is clear that the following scaling property holds
\begin{equation}\label{eq:scalingtrace}
\Tr^k (u(\lambda \cdot)) = \lambda^k (\Tr^k u)(\lambda\cdot),  \  \ u\in F^s_{p,q}(I,w_\g;X), \  \lambda>0.
\end{equation}

\begin{proof}
Let us first consider the case $I=\R$.
Pick $\tilde{\g} \in (-1,p-1)$ with $\tilde{\g} < \gamma$ and set $\tilde{s}:=s+\frac{\tilde{\g}-\g}{p}$. Then, by the Sobolev embedding Theorem \ref{t:sobolevTriebel} and the elementary embedding~\eqref{eq:EmbElementary},
$$
F^s_{p,q}(\R,w_\g;X) \hookrightarrow F^{\tilde{s}}_{p,1}(\R,w_{\tilde{\g}};X) \hookrightarrow
H^{\tilde{s},p}(\R,w_{\tilde{\g}};X).
$$
Combining this with \cite[Proposition~6.3]{LMV18}, we obtain that
$$
\Tr^k:F^s_{p,q}(\R,w_\g;X) \cap C^k(\R;X) \to X, \qquad u \mapsto \partial^k_t u\,(0),
$$
has an extension to a bounded linear operator $\Tr^k:F^s_{p,q}(\R,w_\g;X) \to X$.
Given $\sigma \in (k+\frac{1+\g}{p},s)$, we have $F^s_{p,q}(\R,w_\g;X) \hookrightarrow F^\sigma_{p,1}(\R,w_\g;X)$ and in the same way as above it can be seen that $\Tr^k$ is bounded from $\big(F^s_{p,q}(\R,w_\g;X),\|\,\cdot\,\|_{F^\sigma_{p,1}(\R,w_\g;X)}\big)$ to $X$.
The uniqueness statement thus follows from the density of $F^s_{p,q}(\R,w_\g;X)$ in $F^\sigma_{p,1}(\R,w_\g;X)$.

It remains to treat the case $I=\R_+$.
To this end, pick $\tilde{\g} \in (-1,p-1)$ with $\tilde{\g} < \gamma$ and set $\tilde{s}:=s+\frac{\tilde{\g}-\g}{p}$.
Then, by the Sobolev embedding Theorem \ref{t:sobolevTriebel},
\begin{equation}\label{eq:lemma:t:trace_FF_FL_spaces;X_0;embd_into_B}
F^s_{p,q}(I,w_\g;X) \hookrightarrow F^{\tilde{s}}_{p,p}(I,w_{\tilde{\g}};X).
\end{equation}
By Proposition~\ref{prop:ext} there exists an extension operator
\begin{equation}\label{eq:lemma:t:trace_FF_FL_spaces;X_0;ext}
\Ext :F^{\tilde{s}}_{p,p}(I,w_{\tilde{\g}};X) \to F^{\tilde{s}}_{p,p}(\R,w_{\tilde{\g}};X), \quad \text{with} \quad \Ext C^k(\bar{I};X) \subseteq C^k(\R;X).
\end{equation}
Combining \eqref{eq:lemma:t:trace_FF_FL_spaces;X_0;embd_into_B} and \eqref{eq:lemma:t:trace_FF_FL_spaces;X_0;ext}, the desired result follows from the case $I=\R$.

The case of Bessel potential spaces follows from \eqref{eq:EmbElementary}.
\end{proof}

Our first main result is the following trace embedding.
\begin{theorem}
\label{t:trace_FF_FL_spaces}
Let $I \in \{\R_+,\R\}$ and let $(X_0, X_1)$ be an interpolation couple of Banach spaces. Let $p_0,p_1\in (1, \infty)$, $q_0, q_1\in [1, \infty]$, $\g_0,\g_1\in (-1,\infty)$ and $s_0,s_1\in \R$.
Assume that $s_0-\frac{1+\g_0}{p_0}>k$ and $s_1-\frac{1+\g_1}{p_1}<k$ for some $k\in\N_0$. Set
\begin{equation}
\label{eq:FF_parameters_k}
\delta_i=s_i-\frac{1+\g_i}{p_i},\quad i\in \{0,1\}, \quad \theta=\frac{\delta_0-k}{\delta_0-\delta_1},
\quad\text{and } \quad
\frac{1}{p}=\frac{1-\theta}{p_0}+\frac{\theta}{p_1}.
\end{equation}
Set
\begin{equation}
\label{eq:F_intersected_Triebel_Lizorkin_inhomogeneous_whole_line}
\mF:= F^{s_0}_{p_0,q_0}(I,w_{\g_0};X_0)\cap F^{s_1}_{p_1,q_1}(	I,w_{\g_1};X_1).
\end{equation}
Then the $k$-th order trace operator $\Tr^k:F^{s_0}_{p_0,q_0}(I,w_{\g_0};X_0) \to X_0$ (see Lemma~\ref{lemma:t:trace_FF_FL_spaces;X_0}) acts a bounded linear operator
\begin{equation}
\label{eq:Trk_FF_FL}
\Tr^k: \mF\to (X_0,X_1)_{\theta,p},
\end{equation}
and there is a constant $C$ such that for all $u \in \mF$,
\begin{equation}
\label{eq:boundedness_trace_sharp;inhom}
\|\Tr^k u\|_{(X_0,X_1)_{\theta,p}}
\leq C\|u\|_{F^{s_0}_{p_0,q_0}(I,w_{\g_0};X_0)}^{1-\theta}
\|u\|_{F^{s_1}_{p_1,q_1}(I,w_{\g_1};X_1)}^\theta.
\end{equation}
Moreover, if $s_1=0$ and $\g_1 \in (-1,p_1-1)$
the same holds when $\mF$ is replaced by
\begin{equation}
\label{eq:F_intersected_Triebel_Lizorkin_Lebesgue_inhomogeneous_whole_line}
\mF_L:= F^{s_0}_{p_0,q_0}(I,w_{\g_0};X_0)\cap L^{p_1}(	I,w_{\g_1};X_1)
\end{equation}
and $\|u\|_{F^{s_1}_{p_1,q_1}(I,w_{\g_1};X_1)}$ is replaced by $\|u\|_{L^{p_1}(I,w_{\g_1};X_1)}$ in \eqref{eq:boundedness_trace_sharp;inhom}.
\end{theorem}

The proof of Theorem~\ref{t:trace_FF_FL_spaces} will be given at the end of this subsection. Before embarking into its proof, let us point out some useful facts.

\begin{remark}\
\begin{enumerate}[(i)]
\item Theorem \ref{t:trace_FF_FL_spaces} extends \cite[Theorem 4.1]{MV14}, where only special cases of spaces $X_1$ were considered, $p_0=p_1$, and $\g_0=\g_1$. The argument in the proof of Theorem \ref{t:trace_FF_FL_spaces} is an extension of the arguments used in \cite{MV14}.
\item By \cite[Theorem 1.1]{MV14}, the target space in \eqref{eq:Trk_FF_FL} is optimal in general.
\item The statements of Theorem \ref{t:trace_FF_FL_spaces}  are not optimal in the case that $X_0\hookrightarrow X_1$. Indeed, in that case Lemma \ref{lemma:t:trace_FF_FL_spaces;X_0} already shows that $\Tr^k$ maps into the smaller space $X_0$.
\item The number $\delta:=s-\frac{1+\g}{p}$ is called the \textit{Sobolev index} of the space $\mathcal{A}^{s,p}(\R,w_{\g};X)$, where $\mathcal{A}^{s,p}\in \{F^{s}_{p,q}, B^{s}_{p,q}, H^{s,p}\}$ since this number is invariant under (sharp) Sobolev embeddings (see Theorem \ref{t:sobolevTriebel}). The parameters $\theta,p$ in Theorem \ref{t:trace_FF_FL_spaces} depend only on the corresponding Sobolev index.
\end{enumerate}
\end{remark}

The proof of Theorem \ref{t:trace_FF_FL_spaces} is based on the following key result concerning the intersection of a Triebel-Lizorkin space and a Lebesgue space. Some of the ideas can be traced back to \cite{DiBlasio84}.
The main novelty is that by using Lemma \ref{l:eqreal} we do not need any structural conditions on $X_1$.

\begin{lemma}
\label{l:trace_case_simple}
Let $I \in \{\R_+,\R\}$ and let $(X_0, X_1)$ be an interpolation couple of Banach spaces. Let $s_0\in (0,1)$, $p_0,p_1\in (1, \infty)$, $\g_0\in (-1,p_0-1)$ and  $\g_1\in (-1,p_1-1)$.
Set
\begin{equation}\label{eq:FL_simple_case_parameters}
\theta=\frac{\delta_0}{\delta_0-\delta_1},
\quad
 \delta_0 = s_0-\frac{1+\gamma_0}{p_0},\quad   \delta_1=-\frac{1+\g_1}{p_1},
 \quad   \text{and} \quad
 \frac{1}{p}=\frac{1-\theta}{p_0}+\frac{\theta}{p_1}.
\end{equation}
Assume that $\delta_0>0$. Then for all $u \in F^{s_0}_{p_0,p_0}(I,w_{\g_0};X_0) \cap L^{p_1}(I,w_{\g_1};X_1)$,
\begin{equation}
\label{eq:estimate_u_basic_estimate;Tr}
\|\Tr u\|_{(X_0,X_1)_{\theta,p}}\lesssim
\Big([u]_{{F}^{s_0}_{p_0,p_0}(I,w_{\g_0};X_0)}^{(1)}\Big)^{1-\theta}
\Big(\|u\|_{L^{p_1}(I,w_{\g_1};X_1)}\Big)^\theta,
\end{equation}
where $[u]_{{F}^{s_0}_{p_0,p_0}(I,w_{\g_0};X_0)}^{(1)}$ is as in \eqref{eq:defFmnorm} and $\Tr:=\Tr^0$ (see Lemma~\ref{lemma:t:trace_FF_FL_spaces;X_0}).
\end{lemma}
By Lemma \ref{lemma:t:trace_FF_FL_spaces;X_0}, $\Tr u$ exists in $X_0$, and coincides with $u(0)$ if $u$ is continuous.

\begin{proof}
Let us note that $\delta_1 \in (-1,0)$ as $\g_1 \in (-1,p_1-1)$, which will be important in the computations below.
Below we will use the Hardy-Young inequality, which says that for each measurable function $f:\R_+\to \R_+$, $p\in [1,\infty)$, and $\beta>0$,
\begin{equation}
\label{eq:HardyYoung}
\int_0^{\infty} \sigma^{-\beta p-1} \Big(\int_0^{\sigma}f(s)ds\Big)^{p}d\sigma
\leq \beta^{-p} \int_0^{\infty} \sigma^{-\beta p-1+p} (f(\sigma))^{p} d\sigma.
\end{equation}

To prove \eqref{eq:estimate_u_basic_estimate;Tr} it suffices to consider $I = \R_+$. On the other hand, by an extension argument (see Proposition \ref{prop:ext}), and density (see Lemma \ref{l:density}), we may assume $u\in \S(\R,X_0\cap X_1)$.

In order to prove \eqref{eq:estimate_u_basic_estimate;Tr}, by a scaling argument it suffices to prove that
\begin{equation}
\label{eq:estimate_u_basic_estimate;sum_form}
\|u(0)\|_{(X_0,X_1)_{\theta,p}}\lesssim
[u]_{{F}^{s_0}_{p_0,p_0}(\R_+,w_{\g_0};X_0)}^{(1)}+\|u\|_{L^{p_1}(\R_+,w_{\g_1};X_1)}.
\end{equation}

Indeed, if $u$ is zero this is obvious. If $u$ is nonzero, then $[u]_{{F}^{s_0}_{p_0,p_0}(\R_+,w_{\g_0};X_0)}^{(1)}>0$ since otherwise $u$ would be a constant function, which contradicts $u\in L^{p_1}(\R_+,w_{\g_1};X_1)$.
Applying \eqref{eq:estimate_u_basic_estimate;sum_form} to $u_\lambda=u(\lambda\,\cdot\,)$ with $\lambda >0$, we obtain
\begin{align*}
\|u(0)\|_{(X_0,X_1)_{\theta,p}}
&= \|u_\lambda(0)\|_{(X_0,X_1)_{\theta,p}}\\
&\lesssim [u_\lambda]_{{F}^{s_0}_{p_0,p_0}(\R_+,w_{\g_0};X_0)}^{(1)}+\|u_\lambda\|_{L^{p_1}(\R_+,w_{\g_1};X_1)} \\
&=
\lambda^{\delta_0}[u]_{{F}^{s_0}_{p_0,p_0}(\R_+,w_{\g_0};X_0)}^{(1)}+\lambda^{\delta_1}\|u\|_{L^{p_1}(\R_+,w_{\g_1};X_1)}.
\end{align*}
As $\theta=\frac{\delta_0}{\delta_0-\delta_1}$ and $1-\theta=-\frac{\delta_1}{\delta_0-\delta_1}$, the choice
$$
\lambda = \left(\frac{\|u\|_{L^{p_1}(\R_+,w_{\g_1};X_1)}}{[u]_{{F}^{s_0}_{p_0,p_0}(\R_+,w_{\g_0};X_0)}^{(1)}}\right)^{\frac{1}{\delta_0-\delta_1}}
$$
then leads to \eqref{eq:estimate_u_basic_estimate;Tr}.

To prove \eqref{eq:estimate_u_basic_estimate;sum_form}
we use the standard identity
\begin{equation}
\label{eq:decomposition}
u(0)= \underbrace{\int_0^{\sigma}  t^{-2}\Big( \int_0^{t} (u(\tau)-u(t)) d\tau \Big) dt}_{=:T_0(\sigma)} + \underbrace{\frac{1}{\sigma}\int_0^{\sigma}u(\tau)d\tau}_{=:T_1(\sigma)} ,\   \text{ for each }\sigma> 0.
\end{equation}
To see \eqref{eq:decomposition} note that by elementary calculations one has
\begin{align*}
T_0(\sigma) & = \int_0^{\sigma}  \int_0^t \int_{\tau}^t - \frac{u'(s)}{t^2} d s d\tau dt = \int_0^{\sigma}  \int_0^t - \frac{s u'(s)}{t^2} d s dt
\\ & = \int_0^{\sigma} \int_s^{\sigma}  -\frac{s u'(s)}{t^2} dt  d s
=\int_0^\sigma \Big(\frac{s}{\sigma} - 1\Big) u'(s) d s
  =u(0) - T_1(\sigma).
\end{align*}

To prove \eqref{eq:estimate_u_basic_estimate;sum_form} set $\xi_0:= -\delta_0<0$, $\xi_1:= -\delta_1 >0$ and observe that
$$
\frac{\xi_0}{\xi_0-\xi_1}= \frac{\delta_0}{\delta_0-\delta_1} =\theta.
$$
By \eqref{eq:decomposition} and Lemma \ref{l:eqreal} it suffices to prove
\begin{align}
\label{eq:ineq0}
\|\sigma\mapsto \sigma^{\xi_0} T_0(\sigma)\|_{L^{p_0}(\R_+,\frac{d\sigma}{\sigma};X_0)}
&\lesssim [u]_{{F}^{s_0}_{p_0,p_0}(\R_+,w_{\g_0};X_0)}^{(1),+},\\
\label{eq:ineq1}
\|\sigma\mapsto \sigma^{\xi_1} T_1(\sigma)\|_{L^{p_1}(\R_+,\frac{d\sigma}{\sigma};X_1)}
&\lesssim \|u\|_{L^{p_1}(\R_+,w_{\g_1};X_1)},
\end{align}
where the implicit constants do not depend on $u$. Let us begin by proving \eqref{eq:ineq0}:
\begin{align*}
\Big\|\sigma &\mapsto \sigma^{\xi_0} T_0(\sigma)
\Big\|_{L^{p_0}(\R_+,\frac{d\sigma}{\sigma};X_0)}^{p_0}\\
&\leq\int_0^{\infty} \sigma^{-\delta_0 p_0-1} \Big(\int_0^{\sigma}  t^{-2}\Big( \int_0^{t} \|u(t)-u(\tau)\|_{X_0}  d\tau\Big)dt \Big)^{p_0} d\sigma\\
&\stackrel{\eqref{eq:HardyYoung}}{\lesssim} \int_0^{\infty} \sigma^{-\delta_0 p_0+p_0-1} \sigma^{-2p_0} \left(\int_0^{\sigma} \|u(\sigma)-u(\tau)\|_{X_0}  d\tau\right)^{p_0}d\sigma\\
&= \int_0^{\infty} \sigma^{-s_0 p_0+\g_0-p_0} \Big(\int_{-\sigma}^{0} \|u(\sigma)-u(\sigma+h)\|_{X_0}  dh\Big)^{p_0}d\sigma\\
&\leq \int_0^{\infty} \sigma^{\g_0}\Big[ \sup_{t>0} t^{-s_0 p_0-p_0} \Big(\int_{\max\{-\sigma,-t\}}^{0} \|u(\sigma)-u(\sigma+h)\|_{X_0}  dh\Big)^{p_0}\Big]d\sigma\\
&\stackrel{\eqref{eq:defFmnorm}}{\leq } \Big([u]_{{F}^{s_0}_{p_0,\infty}(\R_+,w_{\g_0};X_0)}^{(1),+}\Big)^{p_0}\lesssim
\Big([u]_{{F}^{s_0}_{p_0,p_0}(\R_+,w_{\g_0};X_0)}^{(1),+}\Big)^{p_0},
\end{align*}
where in the last step we used \eqref{eq:Fmdiscr} and $\ell^{p_0}\hookrightarrow \ell^\infty$. It remains to prove \eqref{eq:ineq1}:
\begin{align*}
\|\sigma \mapsto \sigma^{\xi_1} T_1(\sigma) \|_{L^{p_1}(\R_+,\frac{d\sigma}{\sigma};X_1)}^{p_1}
&\leq \int_0^{\infty} \sigma^{-(\delta_1+1)p_1 -1} \Big( \int_0^{\sigma} \|u(\tau)\|_{X_1}d\tau\Big)^{p_1}   d\sigma\\
&\stackrel{\eqref{eq:HardyYoung}}{\lesssim} \int_0^{\infty} \sigma^{-\delta_1 p_1-1} \|u(\sigma)\|_{X_1}^{p_1} d\sigma\\
&=\|u\|_{L^{p_1}(\R_+,w_{\g_1};X_1)}^{p_1}.
\end{align*}
\end{proof}

We are now ready to prove Theorem \ref{t:trace_FF_FL_spaces}.

\begin{proof}[Proof of Theorem \ref{t:trace_FF_FL_spaces}]
The idea is to reduce the claim to the one proven in Lemma~\ref{l:trace_case_simple} by mixed-derivative and Sobolev embeddings. For the sake of clarity we divide the proof into several steps. In Steps 1-4 we will prove \eqref{eq:boundedness_trace_sharp;inhom} and in Step 5 we will subsequently derive the corresponding statements with $\mF$ replaced by $\mF_L$.

\textit{Step 1}. \emph{Proof of \eqref{eq:boundedness_trace_sharp;inhom} in the case $s_0<1$, $s_1>0$, $k=0$, $q_0=p_0$, $q_1=p_1$, $\g_0 \in (-1,p_0-1)$ and $\g_1\in (-1,p_1-1)$.}
To begin note that $\g_1-s_1 p_1\in (-1,p_1-1)$, where $\g_1-s_1 p_1 < p_1-1$ follows from $\g_1<p_1-1$ and $s_1>0$ and where $\g_1-s_1 p_1 > -1$ follows from $s_1<\frac{1+\g_1}{p_1}$.
By Corollary~\ref{c:t:sobolevTriebel},
\begin{equation}\label{eq:boundedness_trace_sharp;hom;proof;St1_Sob_emb}
F^{s_1}_{p_1,p_1}(I,w_{\g_1};X_1) \hookrightarrow  L^{p_1}(I,w_{\g_1-s_1 p_1};X_1).
\end{equation}
Since $\wt{\delta}_1 := -\frac{1+(\g_1-s_1 p_1)}{p_1} =s_1-\frac{1+\g_1}{p_1} = \delta_1$, it follows from a combination of Lemma \ref{l:trace_case_simple}, and \eqref{eq:boundedness_trace_sharp;hom;proof;St1_Sob_emb} that
\begin{align*}
\|\Tr u\|_{(X_0,X_1)_{\theta,p}}
&\lesssim
\Big([u]_{{F}^{s_0}_{p_0,p_0}(I,w_{\g_0};X_0)}^{(1)}\Big)^{1-\theta}
\Big(\|u\|_{L^{p_1}(I,w_{\g_1-s_1 p_1};X_1)}\Big)^\theta \\
&\lesssim \|u\|_{{F}^{s_0}_{p_0,p_0}(I,w_{\g_0};X_0)}^{1-\theta}\|u\|_{F^{s_1}_{p_1,p_1}(I,w_{\g_1};X_1)}^\theta.
\end{align*}

\textit{Step 2}. \emph{Proof of \eqref{eq:boundedness_trace_sharp;inhom} in the case $s_0<k+1$, $s_1>k$, $q_0=p_0$, $q_1=p_1$, $\g_0 \in (-1,p_0-1)$ and $\g_1\in (-1,p_1-1)$.}
Note that $\Tr^k=\Tr(\partial_t^k \cdot)$ and
\begin{equation}\label{eq:boundedness_trace_sharp;hom;proof;St1_derivative}
\partial_t^k: F^{s_i}_{p_i,p_i}(I,w_{\g_i};X_i) \to  F^{s_i-k}_{p_i,p_i}(I,w_{\g_i};X_i), \qquad i\in \{0,1\},
\end{equation}
as a bounded linear operator (see \cite[Proposition 3.10]{MV12}).
Setting $\tilde{\delta}_i :=(s_i-k)-\frac{1+\g_i}{p_i} = \delta_i-k$ for $i\in \{0,1\}$, we have
\begin{equation}\label{eq:boundedness_trace_sharp;hom;proof;St1_theta}
\wt{\theta}:= \frac{\wt{\delta}_0}{\wt{\delta}_0-\wt{\delta}_1} = \frac{\delta_0-k}{\delta_0-\delta_1}=\theta.
\end{equation}
The desired estimate \eqref{eq:boundedness_trace_sharp;inhom} thus follows from Step 1.

\textit{Step 3}. \emph{Proof of \eqref{eq:boundedness_trace_sharp;inhom} in the case $q_0=p_0$, $q_1=p_1$, $\g_0\in (-1,p_0-1)$ and $\g_1\in (-1,p_1-1)$.}
As $k < \frac{1+\g_0}{p_{0}}+k < s_{0} \wedge (k+1)$ and $s_{1} \vee k < \frac{1+\g_1}{p_{1}} +k  < k+1$, we can choose $\vartheta_{0},\vartheta_{1} \in (0,1)$ such that $k  < \frac{1+\tilde{\g}_0}{\tilde{p}_{0}}+k<\tilde{s}_0<k+1$ and $k<\tilde{s}_1<\frac{1+\tilde{\g}_1}{\tilde{p}_{1}}+k$, where
\begin{equation}
\label{eq:definition_s_g_p_tilde_proof_step_2}
\tilde{s}_{i} := s_{0}(1-\vartheta_{i})+s_{1}\vartheta_{i},\qquad
\frac{1}{\tilde{p}_{i}}= \frac{1-\vartheta_{i}}{p_{0}} + \frac{\vartheta_{i}}{p_{1}},
\qquad \frac{\tilde{\g}_i}{\tilde{p}_i}=\frac{1-\vartheta_i}{p_0}\g_0+\frac{\vartheta_i}{p_1}\g_1.
\end{equation}
To see that such $\vartheta_i$ exist, let us first note that $\delta_0>k$, $\delta_1<k$, $\theta \in (0,1)$ and
\[(1-\theta)\delta_0 + \theta\delta_1 = k.\]
Now taking $\vartheta_0 = \theta-\varepsilon \in (0,1)$ and $\vartheta_1 = \theta+\varepsilon \in (0,1)$ with $\varepsilon>0$ one has
\[\tilde{\delta}_{i} := \tilde{s}_{i} - \frac{1+\tilde{\g}_i}{\tilde{p}_{i}} =  (1-\vartheta_i)\delta_0 +
\vartheta_i \delta_1 = k+(-1)^{i} \varepsilon (\delta_0-\delta_1)\]
Therefore, $\frac{1+\tilde{\g}_0}{\tilde{p}_{0}}+k<\tilde{s}_0$ and $\tilde{s}_1<\frac{1+\tilde{\g}_0}{\tilde{p}_{1}}+k$. On the other hand, since
\[\tilde{s}_i = \frac{1+\tilde{\g}_i}{\tilde{p}_i} + k  + (-1)^{i}\varepsilon (\delta_0-\delta_1)  \]
and $\frac{1+\tilde{\g}_i}{\tilde{p}_i}\in (0,1)$, choosing $\varepsilon>0$ small enough we find that $\tilde{s}_i\in (k,k+1)$.

By \eqref{eq:definition_s_g_p_tilde_proof_step_2}, one can check that $\tilde{\g}_i\in (-1,\tilde{p}_i-1)$ for $i\in \{0,1\}$.
Moreover, $\tilde{\theta}:=(\tilde{\delta}_{0}-k)/(\tilde{\delta}_{0}-\tilde{\delta}_{1})=1/2$ and
\begin{equation}
\label{eq:tilde_parameters_p_q_vartheta}
\begin{aligned}
 \frac{1}{\tilde{p}} := \frac{\tilde{\theta}}{\tilde{p}_{0}} + \frac{1-\tilde{\theta}}{\tilde{p}_{1}}=
\frac{1}{2}\Big( \frac{1}{\tilde{p}_0}+\frac{1}{\tilde{p}_1}\Big)&=\frac{1-\theta}{p_0}+\frac{\theta}{p_1},\\
 (1-\tilde{\theta})\vartheta_{0}+\tilde{\theta}\vartheta_{1}=\frac{1}{2}(\vartheta_{0}+\vartheta_{1})&=\theta.
\end{aligned}
\end{equation}
Applying the mixed-derivative Theorem~\ref{t:interpolationA}, we get that for $i\in \{0,1\}$,
\begin{equation}
\label{e:t:Trace;mixed-derivative}
\|u\|_{F^{\tilde{s}_i}_{\tilde{p}_i,\tilde{p}_i}(I,w_{\tilde{\g}_i};(X_0,X_{1})_{\vartheta_{i},1})}
\lesssim  \\
\|u\|_{F^{s_0}_{p_0,p_0}(I,w_{\g_0};X_{0})}^{1-\vartheta_i}
\|u\|_{F^{s_1}_{p_1,p_1}(I,w_{\g_1};X_{1})}^{\vartheta_i}.
\end{equation}
As
$$
(X_{0},X_{1})_{\theta,p} =(X_{0},X_{1})_{\frac{1}{2}\vartheta_1+\frac{1}{2}\vartheta_2,p} = ((X_0,X_1)_{\vartheta_0,1},(X_0,X_1)_{\vartheta_1,1})_{\frac{1}{2},p}
$$
by the second identity in the second line of \eqref{eq:tilde_parameters_p_q_vartheta} and reiteration \eqref{eq:reit}, an application of Step 2 yields the estimate
\begin{align*}
\|\Tr^k u\|_{(X_{0},X_{1})_{\theta,p}}
\lesssim \|u\|_{F^{\tilde{s}_0}_{\tilde{p}_0,\tilde{p}_0}(I,w_{\tilde{\g}_0};(X_0,X_{1})_{\vartheta_{0},1})}^{1/2}
\|u\|_{F^{\tilde{s}_1}_{\tilde{p}_1,\tilde{p}_1}(I,w_{\tilde{\g}_1};(X_0,X_{1})_{\vartheta_{1},1})}^{1/2}.
\end{align*}
Combining this with \eqref{e:t:Trace;mixed-derivative} and the second identity in the second line of \eqref{eq:tilde_parameters_p_q_vartheta}, we obtain the desired estimate \eqref{eq:boundedness_trace_sharp;inhom}.

\textit{Step 4}. \emph{Proof of \eqref{eq:boundedness_trace_sharp;inhom} in the general case.}
Pick $\tilde{\g}_i \in (-1,p_i-1)$ such that $\tilde{\g}_i < \g_i$ and put $\tilde{s}_i := s_i + \frac{\g_i-\tilde{\g}_i}{p_i}$ for $i\in \{0,1\}$. By the Sobolev embedding Theorem \ref{t:sobolevTriebel},
$$
F^{s_i}_{p_i,q_i}(I,w_{\g_i};X_i)\hookrightarrow F^{\tilde{s}_i}_{p_i,p_i}(I,w_{\tilde{\g}_i};X_i), \qquad i\in \{0,1\}.
$$
The estimate \eqref{eq:boundedness_trace_sharp;inhom} thus follows from Step 2 by noticing that $\tilde{\delta}_i := \tilde{s}_i-\frac{1+\tilde{\g}_i}{p_i} = \delta_i$ where $i\in \{0,1\}$ for the Sobolev indices under the above embeddings, in particular,
$$
\frac{\tilde{\delta}_0}{\tilde{\delta}_1-\tilde{\delta}_0}= \frac{\delta_0}{\delta_1-\delta_0}=\theta.
$$

\textit{Step 5}. \emph{Proof of the last statement}.
As a consequence of \eqref{eq:EmbElementary} we have
$$
\mF_L \hookrightarrow F^{s_0}_{p_0,q_0}(I,w_{\g_0};X_0)
\cap F^{0}_{p_1,\infty}(I,w_{\g_1};X_1),
$$
from which the desired statement follows.
\end{proof}

\subsection{Embedding into Spaces of Continuous Functions}\label{s:restriction}
Next we prove that under the condition $\gamma_0, \gamma_1\geq 0$, the functions $u\in \mF$ are actually continuous with values in $(X_0,X_1)_{\theta,p}$. The argument is by a translation argument, which however is not completely standard in Triebel--Lizorkin spaces. For a Banach space $X$, $C_{\rm b}([0,\infty);X)$ denotes the space of all bounded and continuous maps $u:[0,\infty)\to X$ endowed with the norm
$\|u\|_{C_{\rm b}([0,\infty);X)}:=\sup_{t\in [0,\infty)}\|u(t)\|_X.$

\begin{theorem}
\label{t:embeddings_continuous_function_spaces}
Let $I \in \{\R_+,\R\}$ and let $(X_0, X_1)$ be an interpolation couple of Banach spaces. Let $p_0,p_1\in (1, \infty)$, $q_0, q_1\in [1, \infty]$, $\g_0,\g_1\in [0,\infty)$ and $s_0,s_1\in \R$. Assume that $s_0-\frac{1+\g_0}{p_0}>k$ and $s_1-\frac{1+\g_1}{p_1}<k$ for some $k\in\N_0$. Let $\theta,\delta_0,\delta_1,p$ and $\mF$ be as in \eqref{eq:FF_parameters_k} and \eqref{eq:F_intersected_Triebel_Lizorkin_inhomogeneous_whole_line}, respectively.
Then the $k$-th order derivative operator $\partial^k_t$ has the mapping property
$$
\partial^k_t:\mF \to C_{\mathrm{b}}(\bar{I};(X_0,X_1)_{\theta,p})
$$
and for all $u \in \mF$
\begin{equation}\label{eq:t:embeddings_continuous_function_spaces}
\|\partial^k_t u\|_{C_{\mathrm{b}}(\bar{I};(X_0,X_1)_{\theta,p})}
\lesssim \|u\|_{F^{s_0}_{p_0,q_0}(I,w_{\g_0};X_0)}^{1-\theta}\|u\|_{F^{s_1}_{p_1,q_1}(I,w_{\g_1};X_1)}^{\theta}.
\end{equation}
Moreover, for $s_1=0$ and $\g_1 \in (-1,p_1-1)$ the same holds true when $\mF$ is replaced by $\mF_L$, where $\mF_L$ is as in \eqref{eq:F_intersected_Triebel_Lizorkin_Lebesgue_inhomogeneous_whole_line}, and $\|u\|_{F^{s_1}_{p_1,q_1}(\R,w_{\g_1};X_1)}$ is replaced by $\|u\|_{L^{p_1}(\R,w_{\g_1};X_1)}$ in \eqref{eq:t:embeddings_continuous_function_spaces}.
\end{theorem}

We will use the following lemma in the proof of Theorem~\ref{t:embeddings_continuous_function_spaces}.

\begin{lemma}\label{lem:t:embeddings_continuous_function_spaces} Let $I \in \{\R_+,\R\}$, $X$ a Banach space, $p \in (1,\infty)$, $q \in [1,\infty]$ and $s \in \R$. Then the left translation semigroup $T:[0,\infty) \to \calL(\D'(I;X))$ restricts to a $C_0$-semigroup $T:[0,\infty) \to \calL(F^s_{p,q}(I;X))$ of contractions.
\end{lemma}

For the above lemma in case $q = \infty$ it is important that $F^s_{p,\infty}(\R;X)$ is defined in terms of the iterated Bochner space $L^p(\R;\ell^\infty(\N_0;X))$ instead of the mixed-norm Bochner space $L^p(\R)[\ell^\infty(\N_0)](X)$. We do not know whether these two coincide.
\begin{proof}
As the case $I=\R_+$ follows easily from the case $I=\R$, we only consider the latter.
It will be convenient to write
$$
\ell^q_s(\N_0;X) = \left\{ (x_n)_{n \in \N_0} \in X^{\N_0} : (2^{ns}x_{n})_{n} \in \ell^q(\N_0;X) \right\}.
$$
Let $f \in F^s_{p,q}(\R;X)$. Then, for all $t \geq 0$,
\begin{align*}
\|T(t)f\|_{F^s_{p,q}(\R;X)}
&= \|(S_nT(t)f)_{n}\|_{L^p(\R;\ell^q_s(\N_0;X))} = \|T(t)(S_nf)_{n}\|_{L^p(\R;\ell^q_s(\N_0;X))} \\
&= \|(S_nf)_{n}\|_{L^p(\R;\ell^q_s(\N_0;X))}
= \|f\|_{F^s_{p,q}(\R;X)}
\end{align*}
as $T(t)$ is an isometry on $L^p(\R;\ell^q_s(\N_0;X))$.
Furthermore,
\begin{align*}
\|T(t)f-f\|_{F^s_{p,q}(\R;X)}
= \|(T(t)-I)(S_nf)_{n}\|_{L^p(\R;\ell^q_s(\N_0;X))} \to 0 \quad \text{as} \quad t \searrow 0
\end{align*}
by the strong continuity of $T$ on $L^p(\R;\ell^q_s(\N_0;X))$.
\end{proof}

\begin{proof}[Proof of Theorem~\ref{t:embeddings_continuous_function_spaces}]
\textit{Step 1}. \emph{Proof of \eqref{eq:t:embeddings_continuous_function_spaces} in the case $q_0=q_1=\infty$, $\g_0=\g_1=0$.}
Let $u \in \mF$. Then, by Theorem~\ref{t:trace_FF_FL_spaces} and Lemma~\ref{lem:t:embeddings_continuous_function_spaces}, $v(t):= \Tr^kT(t)u$ defines a function $v \in C_{\mathrm{b}}(\bar{I};(X_0,X_1)_{\theta,p})$ with
\begin{equation*}
\|v\|_{C_{\mathrm{b}}(\bar{I};(X_0,X_1)_{\theta,p})}
\lesssim \|u\|_{F^{s_0}_{p_0,\infty}(I;X_0)}^{1-\theta}\|u\|_{F^{s_1}_{p_1,\infty}(I;X_1)}^{\theta}.
\end{equation*}
By \cite[Proposition 7.4]{MV12}, $u\in C^k(\overline{I};X_0)$. Therefore, by Lemma \ref{lemma:t:trace_FF_FL_spaces;X_0}, $v(t) = u(t)$ in $X_0+X_1$ for all $t\geq 0$, and thus $u=v$, which completes the proof of this step.

\textit{Step 2}. \emph{Proof of \eqref{eq:t:embeddings_continuous_function_spaces} in the general case.}
This can be derived from Step 1 by the Sobolev embedding Theorem \ref{t:sobolevTriebel} as in Step 4 of the proof of Theorem~\ref{t:trace_FF_FL_spaces}.

\textit{Step 3}. \emph{Proof of the last statement.} This follows from the \eqref{eq:EmbElementary} in the same way as in Step 5 of the proof of Theorem~\ref{t:trace_FF_FL_spaces}.
\end{proof}

Next we take advantage of the special structure of the weights to obtain \textit{instantaneous regularization} in anisotropic function spaces.
The weight $w_{\g}$ essentially only acts in $0$. Therefore, away from $0$, the functions in corresponding weighted spaces are smoother. This simple idea is developed in the next result. Such results are well-known and available in several special cases in the literature (cf. \cite{AV20,MeySchnau12b,MeySchn12,PruSim04,pruss2016moving}). The next result unifies them.
The following theorem is most conveniently formulated in terms of weighted $C_{\rm{b}}$-spaces. Given a Banach space $X$ and a weight parameter $\mu \in \R$, we define
\begin{equation}\label{eq:Cbmu}
C_{\rm{b},\mu}(\R_+;X) := \left\{ u \in C(\R_+;X) : [t \mapsto t^{\mu}u(t)] \ \text{is bounded and continuous}  \right\}
\end{equation}
endowed with the norm
$\|u\|_{C_{\rm{b},\mu}(\R_+;X)}:=\sup_{t>0} t^{\eta}\|u(t)\|_{X}$. As above $\R_+=(0,\infty)$.

\begin{theorem}[Instantaneous regularization]
\label{t:regularization_F}
Let $(X_0, X_1)$ be an interpolation couple of Banach spaces. Let $p_0,p_1\in (1, \infty)$, $q_0, q_1\in [1, \infty]$, $\g_0,\g_1\in [0,\infty)$ and $s_0,s_1\in \R$.
Assume that $s_0-\frac{1}{p_0}>k$ and $s_1-\frac{1}{p_1}<k$ for some $k\in\N_0$. Set
\begin{equation}
\label{eq:FF_parameters_regularization}
\begin{aligned}
&\eta:=\frac{\beta_0-k}{\beta_0-\beta_1},
\quad
\beta_i:=s_i-\frac{1}{p_i},\quad i\in\{0,1\},\\
 & \frac{1}{r}=\frac{1-\eta}{p_0}+\frac{\eta}{p_1} \quad \text{and } \quad \mu:= \frac{1-\eta}{p_0}\g_0+\frac{\eta}{p_1}\g_1.
\end{aligned}
\end{equation}
Let $\mF$ be as in \eqref{eq:F_intersected_Triebel_Lizorkin_inhomogeneous_whole_line} for $I=\R_+$. Then $\partial^k_t : \mF \to C_{\mathrm{b},\mu}(\R_+;(X_0,X_1)_{\eta,r})$
and for all $u \in \mF$
\begin{equation}\label{eq:boundedness_restr_sharp_regularization;inhom;weighted_est}
\|\partial_t^k u\|_{C_{\mathrm{b},\mu}(\R_+;(X_0,X_1)_{\eta,r})}
\lesssim
\|u\|_{F^{s_0}_{p_0,q_0}(\R_+,w_{\g_0};X_0)}^{1-\eta}\|u\|_{F^{s_1}_{p_1,q_1}(\R_+,w_{\g_1};X_1)}^{\eta}.
\end{equation}
Moreover, for $s_1=0$ and $\g_1 \in [0,p_1-1)$ the same holds true when $\mF$ is replaced by $\mF_L$, where $\mF_L$ is as in \eqref{eq:F_intersected_Triebel_Lizorkin_Lebesgue_inhomogeneous_whole_line}, and $\|u\|_{F^{s_1}_{p_1,q_1}(\R,w_{\g_1};X_1)}$ is replaced by $\|u\|_{L^{p_1}(\R,w_{\g_1};X_1)}$ in \eqref{eq:boundedness_restr_sharp_regularization;inhom;weighted_est}.
\end{theorem}

\begin{proof}
Using Theorem~\ref{thm:local_mean_Rychkov} (or the difference norm of Proposition \ref{prop:equivalentNormF} in the $A_p$-case $\gamma \in [0,p-1)$) it follows that for any Banach space $X$ and for every $\varepsilon>0$ the restriction operator $\Restrepsilon: u\mapsto u|_{[\varepsilon,\infty)}$
maps $F^s_{p,q}(\R_+,w_\g;X)$ into $F^s_{p,q}((\varepsilon,\infty);X)$
and
\[\|\Restrepsilon u\|_{F^s_{p,q}((\varepsilon,\infty);X)}
\lesssim \varepsilon^{-\g/p} \|u\|_{F^s_{p,q}(\R_+,w_\g;X)}, \ \ \ u \in F^s_{p,q}(\R_+,w_\g;X).\]
Therefore, by Theorem~\ref{t:embeddings_continuous_function_spaces} and a translation argument, for every $u \in \mF$ and $\varepsilon > 0$,
\begin{align*}
\|\Restrkepsilon u\|_{C_{\mathrm{b}}([\varepsilon,\infty);(X_0,X_1)_{\eta,r})}
&\lesssim \|\Restrepsilon u\|_{F^{s_0}_{p_0,q_0}((\varepsilon,\infty);X_0)}^{1-\eta}
\|\Restrepsilon u\|_{F^{s_1}_{p_1,q_1}((\varepsilon,\infty);X_1)}^{\eta} \\
&\lesssim \varepsilon^{-\mu}\|u\|_{F^{s_0}_{p_0,q_0}(\R_+,w_{\g_0};X_0)}^{1-\eta}
\|u\|_{F^{s_1}_{p_1,q_1}(\R_+,w_{\g_1};X_1)}^{\eta}.
\end{align*}
The estimate
\eqref{eq:boundedness_restr_sharp_regularization;inhom;weighted_est} follows from this.
The corresponding estimate with $\mF$ replaced by $\mF_L$ subsequently follows from \eqref{eq:EmbElementary} in the same way as in Step~5 of the proof of Theorem~\ref{t:trace_FF_FL_spaces}.
\end{proof}

\begin{remark}\label{r:regularization}
In applications to evolution equations, one typically has $\g_i>0$ for some $i\in \{0,1\}$ and $X_1\hookrightarrow X_0$. In such a case, any $u\in \mF$ \textit{instantaneously regularizes in space}, i.e. for all $\varepsilon>0$
$$
u(\varepsilon)\in (X_0,X_1)_{\eta,r},
\quad \text{ while }\quad
u(0)\in (X_0,X_1)_{\theta,p},
$$
Since $\g_i>0$, one can check that $\eta>\theta$, and therefore by \cite[Theorem 3.4.1]{BeLo}, $(X_0,X_1)_{\eta,r}\hookrightarrow (X_0,X_1)_{\theta,p}$, where the inclusion is strict in general.

The above continuity results will be applied (also in situations with $X_1\not\hookrightarrow X_0$) in Examples \ref{ex:double_fractional} and \ref{ex:stochastic_fractional}.
\end{remark}

\subsection{Consequences in case of homogeneous norms}\label{sec:homF}

In this section we present homogeneous versions of Theorem~\ref{t:trace_FF_FL_spaces}, Theorem~\ref{t:embeddings_continuous_function_spaces} and Theorem~\ref{t:regularization_F} under the additional assumptions that $s_1>0$ and $\gamma_i<p_i-1$ for $i\in \{1, 2\}$. Recall that the seminorms $[u]_{F^{s}_{p,q}(I,w_{\g};X)}$ are defined in \eqref{eq:defFmnorm} and Remark \ref{rem:seminorms}.

\begin{theorem}
\label{t:trace_FF_FL_spaces;hom}
Let $I \in \{\R_+,\R\}$ and let $(X_0, X_1)$ be an interpolation couple of Banach spaces. Let $p_i\in (1, \infty)$, $q_i\in [1, \infty]$, $\g_i\in (-1,p_i-1)$ and $s_i>0$ for $i\in \{0,1\}$.
Assume that $s_0-\frac{1+\g_0}{p_0}>k$ and $s_1-\frac{1+\g_1}{p_1}<k$ for some $k\in\N_0$.
Let $\theta,\delta_0,\delta_1,p$ be as in \eqref{eq:FF_parameters_k}.
Then for all $u \in F^{s_0}_{p_0,q_0}(I,w_{\g_0};X_0)\cap F^{s_1}_{p_1,q_1}(I,w_{\g_1};X_1)$,
\begin{equation}
\label{eq:boundedness_trace_sharp}
\|\Tr^k u\|_{(X_0,X_1)_{\theta,p}}
\lesssim [u]_{F^{s_0}_{p_0,q_0}(I,w_{\g_0};X_0)}^{1-\theta}
[u]_{F^{s_1}_{p_1,q_1}(I,w_{\g_1};X_1)}^\theta.
\end{equation}
Moreover, for $s_1=0$ the same holds true when
$F^{s_1}_{p_1,q_1}(I,w_{\g_1};X_1)$ and $[u]_{F^{s_1}_{p_1,q_1}(I,w_{\g_1};X_1)}$
are replaced by $L^{p_1}(I,w_{\g_1};X_1)$ and $\|u\|_{L^{p_1}(I,w_{\g_1};X_1)}$ in
the above.
\end{theorem}
\begin{proof}
We could simply repeat the arguments in Steps 1-3 from the proof of Theorem~\ref{t:trace_FF_FL_spaces}, where we
use the homogeneous versions \eqref{eq:sob_embedding_Triebel_seminorm}, \eqref{eq:c:sob_embedding_Triebel_seminorm}, \eqref{eq:mixed_derivative_seminorm} of the Sobolev embedding Theorem \ref{t:sobolevTriebel}, the Sobolev embedding Corollary~\ref{c:t:sobolevTriebel} and the mixed-derivative Theorem~\ref{t:interpolationA}, respectively, and the homogeneous version of \eqref{eq:boundedness_trace_sharp;hom;proof;St1_derivative}. To be more explicit, the homogeneous version of \eqref{eq:boundedness_trace_sharp;hom;proof;St1_derivative} is
\begin{equation*}
[\partial_t^k u]_{F^{s_i-k}_{p_i,p_i}(I,w_{\g_i};X_i)} \lesssim
[u]_{F^{s_i}_{p_i,p_i}(I,w_{\g_i};X_i)}, \ \  i \in \{0,1\}, \ u\in F^{s_i}_{p_i,p_i}(I,w_{\g_i};X_i)
\end{equation*}
and follows from \eqref{eq:boundedness_trace_sharp;hom;proof;St1_derivative} by a standard scaling argument.

However, we will derive \eqref{eq:boundedness_trace_sharp} from \eqref{eq:boundedness_trace_sharp;inhom} by a scaling argument as in Lemma \ref{l:trace_case_simple}. To this end, consider again $u_\lambda := u(\lambda\,\cdot\,)$ for $\lambda > 0$. By \eqref{eq:scalingtrace}, \eqref{eq:boundedness_trace_sharp;inhom}, the norm equivalence of Proposition \ref{prop:equivalentNormF} applied to $u_\lambda$ we obtain
\begin{align*}
\lambda^k\|\Tr^k u\|_{(X_0,X_1)_{\theta,p}}
&= \|\Tr^k u_\lambda\|_{(X_0,X_1)_{\theta,p}} \\
&\lesssim \|u_\lambda\|_{F^{s_0}_{p_0,q_0}(I,w_{\g_0};X_0)}^{1-\theta}
\|u_\lambda\|_{F^{s_1}_{p_1,q_1}(I,w_{\g_1};X_1)}^\theta \\
&\eqsim \left(\|u_\lambda\|_{L^{p_0}(I,w_{\g_0};X_0)}+[u_\lambda]_{F^{s_0}_{p_0,q_0}(I,w_{\g_0};X_0)}\right)^{1-\theta} \\
&\qquad \cdot \quad \left(\|u_\lambda\|_{L^{p_1}(I,w_{\g_1};X_1)}+[u_\lambda]_{F^{s_1}_{p_1,q_1}(I,w_{\g_1};X_1)}\right)^{\theta} \\
&= \left(\lambda^{-\frac{1+\g_0}{p_0}}\|u\|_{L^{p_0}(I,w_{\g_0};X_0)}+\lambda^{\delta_0}[u]_{F^{s_0}_{p_0,q_0}(I,w_{\g_0};X_0)}\right)^{1-\theta} \\
&\qquad \cdot \quad \left(\lambda^{-\frac{1+\g_1}{p_1}}\|u\|_{L^{p_1}(I,w_{\g_1};X_1)}+\lambda^{\delta_1}[u]_{F^{s_1}_{p_1,q_1}(I,w_{\g_1};X_1)}\right)^{\theta}.
\end{align*}
Since $\delta_0(1-\theta)+\delta_1\theta=k$ and $-\frac{1+\g_i}{p_i}-\delta_i = -s_i$ for $i\in \{0,1\}$, it follows that
\begin{align*}
\|\Tr^k u\|_{(X_0,X_1)_{\theta,p}}
&\lesssim \left(\lambda^{-s_0}\|u\|_{L^{p_0}(\R,w_{\g_0};X_0)}+[u]_{F^{s_0}_{p_0,q_0}(\R,w_{\g_0};X_0)}\right)^{1-\theta} \\
&\qquad  \cdot \quad \left(\lambda^{-s_1}\|u\|_{L^{p_1}(I,w_{\g_1};X_1)}+[u]_{F^{s_1}_{p_1,q_1}(I,w_{\g_1};X_1)}\right)^{\theta}.
\end{align*}
As $s_0,s_1 > 0$, taking the limit $\lambda \to \infty$ gives the desired estimate \eqref{eq:boundedness_trace_sharp}.

With a slight modification of the above scaling argument, the final assertion follows as well.
\end{proof}

In the same way the next results can be derived from Theorems~\ref{t:embeddings_continuous_function_spaces} and \ref{t:regularization_F}.

\begin{theorem}
\label{t:embeddings_continuous_function_spaces;hom}
Let $I \in \{\R_+,\R\}$ and let $(X_0, X_1)$ be an interpolation couple of Banach spaces. Let $p_i\in (1, \infty)$, $q_i\in [1, \infty]$, $\g_i\in [0,p_i-1)$ and $s_i>0$ for $i\in \{0,1\}$. Assume that $s_0-\frac{1+\g_0}{p_0}>k$ and $s_1-\frac{1+\g_1}{p_1}<k$ for some $k\in\N_0$. Let $\theta,\delta_0,\delta_1,p$ as in \eqref{eq:FF_parameters_k}, respectively.
Then for all $u \in F^{s_0}_{p_0,q_0}(I,w_{\g_0};X_0)\cap F^{s_1}_{p_1,q_1}(I,w_{\g_1};X_1)$,
\begin{equation*}
\|\partial^k_t u\|_{C_{\mathrm{b}}(\bar{I};(X_0,X_1)_{\theta,p})}
\lesssim [u]_{F^{s_0}_{p_0,q_0}(I,w_{\g_0};X_0)}^{1-\theta}[u]_{F^{s_1}_{p_1,q_1}(I,w_{\g_1};X_1)}^{\theta}.
\end{equation*}
Moreover, for $s_1=0$ the same holds true when $F^{s_1}_{p_1,q_1}(I,w_{\g_1};X_1)$ and $[u]_{F^{s_1}_{p_1,q_1}(I,w_{\g_1};X_1)}$
are replaced by $L^{p_1}(I,w_{\g_1};X_1)$ and $\|u\|_{L^{p_1}(I,w_{\g_1};X_1)}$ in
the above.
\end{theorem}

\begin{theorem}[Instantaneous regularization]
\label{t:regularization_F;hom}
Let $(X_0, X_1)$ be an interpolation couple of Banach spaces. Let $p_i\in (1, \infty)$, $q_i\in [1, \infty]$, $\g_i\in [0,p_i-1)$ and $s_i>0$ for $i\in \{0,1\}$.
Assume that $s_0-\frac{1}{p_0}>k$ and $s_1-\frac{1}{p_1}<k$ for some $k\in\N_0$.
Let $\eta, \beta_0, \beta_1, r, \mu$ be as in \eqref{eq:FF_parameters_regularization}.
Then there exists a constant $C$ such that for all $u\in F^{s_0}_{p_0,q_0}(I,w_{\g_0};X_0)\cap F^{s_1}_{p_1,q_1}(I,w_{\g_1};X_1)$,
\begin{equation*}
\|\partial_t^k u\|_{C_{\rm{b},\mu}(\R_+;(X_0,X_1)_{\eta,r})}
\lesssim [u]_{F^{s_0}_{p_0,q_0}(\R_+,w_{\g_0};X_0)}^{1-\eta}[u]_{F^{s_1}_{p_1,q_1}(\R_+,w_{\g_1};X_1)}^\eta.
\end{equation*}
Moreover, for $s_1=0$ the same holds true when $F^{s_1}_{p_1,q_1}(\R_+,w_{\g_1};X_1)$ and $[u]_{F^{s_1}_{p_1,q_1}(\R_+,w_{\g_1};X_1)}$
are replaced by $L^{p_1}$ and $\|u\|_{L^{p_1}(\R_+,w_{\g_1};X_1)}$ in the above.
\end{theorem}

\section{Trace Embeddings for Bessel potential spaces}\label{s:traceBessel}

\subsection{The inhomogeneous case}

In view of the elementary embedding \eqref{eq:EmbElementary}, we obtain the following corollary to Theorem \ref{t:trace_FF_FL_spaces}.

\begin{corollary}\label{c:t:trace_FF_FL_spaces;H}
Let $I \in \{\R_+,\R\}$ and let $(X_0, X_1)$ be an interpolation couple of Banach spaces. Let $p_0,p_1\in (1, \infty)$, $\g_0 \in (-1,p_0-1)$, $\g_1 \in (-1,p_1-1)$ and $s_0,s_1\in \R$.
Assume that $s_0-\frac{1+\g_0}{p_0}>k$ and $s_1-\frac{1+\g_1}{p_1}<k$ for some $k\in\N_0$.
Let $\theta,\delta_0,\delta_1,p$ be as in \eqref{eq:FF_parameters_k} and set
\begin{equation}
\label{eq:F_intersected_Triebel_Lizorkin_inhomogeneous_whole_line;H}
\H:= H^{s_0,p_0}(I,w_{\g_0};X_0)\cap H^{s_1,p_1}(	I,w_{\g_1};X_1).
\end{equation}
Then the $k$-th order trace operator $\Tr^k:H^{s_0,p_0}(I,w_{\g_0};X_0) \to X_0$ (see Lemma~\ref{lemma:t:trace_FF_FL_spaces;X_0}) acts a bounded linear operator
\begin{equation}
\label{eq:Trk_FF_FL;H}
\Tr^k: \H \to (X_0,X_1)_{\theta,p},
\end{equation}
and  for all $u \in \H$
%there exists a constant $C$ such that for all $u \in \H$,
\begin{equation}
\label{eq:boundedness_trace_sharp;hom;H}
\|\Tr^k u\|_{(X_0,X_1)_{\theta,p}}
\lesssim \|u\|_{H^{s_0,p_0}(I,w_{\g_0};X_0)}^{1-\theta}
\|u\|_{H^{s_1,p_1}(I,w_{\g_1};X_1)}^\theta.
\end{equation}
\end{corollary}

Similarly we obtain the following corollary to Theorems \ref{t:embeddings_continuous_function_spaces} and \ref{t:regularization_F}.
\begin{corollary}\label{c:t:trace_FF_FL_spaces;Hcont}
Let $\gamma_i\in [0,p_i-1)$ for $i\in \{0,1\}$. Then both Theorems \ref{t:embeddings_continuous_function_spaces} and \ref{t:regularization_F} hold with $\mF$ and $F^{s_i}_{p_i,q_i}$ replaced by $\H$ and $H^{s_i,p_i}$ respectively.
\end{corollary}
The latter extends \cite[Corollary 7.6]{AV20}, where the result was proved under geometric restrictions on $X_0$ and $X_1$, and in the case $X_1$ was the domain of a fractional power of a sectorial operator on $X_0$.

In particular, Theorem \ref{thmintro:t:trace_FF_FL_spaces;H} for $\mathcal{A} = H$ follows from Corollaries \ref{c:t:trace_FF_FL_spaces;H} and \ref{c:t:trace_FF_FL_spaces;Hcont}. The case $\mathcal{A} = W$ follows similarly, since \eqref{eq:EmbElementary} also holds in that case. The case $\mathcal{A} = F$ was already proved in Theorems \ref{t:trace_FF_FL_spaces}, \ref{t:embeddings_continuous_function_spaces} and \ref{t:regularization_F}.

\subsection{Consequences in case of homogeneous norms}

In order to have a presentation which follows the Triebel-Lizorkin case, we set
\[[u]_{H^{s,p}(\R_+,w_{\gamma};X)} : =\|\Dert^s u\|_{L^p(\R_+,w_{\gamma};X)},\]
where $\Dert$ is as in \eqref{eq:defB}. In Proposition \ref{prop:Hinftyder} conditions are discussed under which one can replace $\Dert$ by $\Der$.

The next result is immediate from Lemma \ref{lem:sandwichHRplus}, and moreover by Remark \ref{rem:generalweights} the Bessel potential version of  Theorem
\ref{t:trace_FF_FL_spaces;hom}  actually holds for $\gamma_i<0$ as well.
\begin{corollary}\label{c:t:trace_FF_FL_spaces;H-hom}
Let $\gamma_i\in [0,p_i-1)$ for $i\in \{0,1\}$. Theorems
\ref{t:trace_FF_FL_spaces;hom}, \ref{t:embeddings_continuous_function_spaces;hom} and \ref{t:regularization_F;hom} hold with $F^{s_i}_{p_i,q_i}$ and $[u]_{F^{s_i}_{p_i,q_i}}$ replaced by $H^{s_i,p_i}$ and $[u]_{H^{s_i,p_i}}$,
respectively.
\end{corollary}

Similar results hold with $H$ replaced by $W$, but they follow more directly by using the seminorms $[u]_{W^{s,p}}:=[u]_{F^{s}_{p,p}}$ if $s\notin \N_0$, and
$[u]_{W^{s,p}}:=\|\partial_t^s u\|_{L^p}$ if $s\in \N_0$.

\section{Applications to evolution equations}
\label{s:applications}
In this section we present applications of the above theory to fractional and stochastic evolution equations. Here the main novelty is that using our trace estimates we can provide sharp estimates on $\R_+$ whereas, these were only available on $(0,T)$ before. In a follow-up paper we plan to use these estimates to derive new a priori estimates for quasi-linear (stochastic) PDEs considered on homogenous function spaces with critical scaling (see \cite{AV19_QSEE_1, CriticalQuasilinear}).

\subsection{Fractional evolution equations}
\label{ss:fractional_evolution_equations} Equations of fractional type
arise in many physical applications and it is the basic model for anomalous diffusion, and we refer the reader to \cite{CLS04,F90_1,H00,KSVZ16,P93_evolutionary,Zacher05} for more details.
Evolution equations of fractional type fit in the framework of Volterra integral equations.  Introductions into Volterra equations can be found in \cite{Bazhlekova_2001_dissertation,GLS,P93_evolutionary,Zacher05}.

In this section we will focus on a regularity problem for the fractional evolution equation on a Banach space $X$:
\begin{equation}\label{eq:fractionalintro}
\partial_t^{\alpha} u + Au=f \ \ \text{on $\R_+$},
\end{equation}
where $\alpha\in (0,2)$ and $A$ is a sectorial operator on $X$. For simplicity we only consider Dirichlet boundary conditions: $u(0) = 0$, and $u'(0) = 0$ if $\alpha>1$.

Eq.\ \eqref{eq:fractionalintro} is extensively studied in the literature \cite{P93_evolutionary, Pruss19, Zacher05}, where actually  $\partial_t^{\alpha}$ is replaced by a more general operator. We will focus on the latter case and only consider zero initial values for simplicity.

In \cite{Pruss19} weighted $L^p$-regularity was studied for \eqref{eq:fractionalintro} under the assumption that $0\in \rho(A)$. Moreover, trace estimates for the solution were provided under the assumption that $A^{1/\alpha}$ is sectorial. In this section we will use our results to remove the conditions that $0\in \rho(A)$ and that $A^{1/\alpha}$ is sectorial.

Motivated by \eqref{eq:DerTinverseid} we say that $u$ is a {\em strong solution} to \eqref{eq:fractionalintro} with initial value zero if
\begin{equation}
\label{eq:fractional_evolution}
u(t)+ \mathcal{K}_{\alpha}*A u(t) = \mathcal{K}_{\alpha}*f(t), \quad \text{\text{a.e.} \ $t\geq 0$},
\end{equation}
where $*$ denotes the convolution on $\R$, and where we use the zero extensions of $Au$ and $f$ on $(-\infty, 0)$ and $\mathcal{K}_{\alpha}(t):=\frac{|t|^{\alpha-1} \one_{(0,\infty)}(t)}{\Gamma(\alpha)}$ (see \eqref{eq:DerTinverseid}).

Below we use the more suggestive notation $\partial_t^\alpha = \Der^{\alpha}$, where $\Der$ is as in \eqref{eq:defC}. Our main result concerning \eqref{eq:fractionalintro} or \eqref{eq:fractional_evolution} is:
\begin{theorem}
\label{t:fractional}
Let $X$ be a UMD space. Let $\alpha\in (0,2)$, $p\in (1,\infty)$ and $\g\in [0,p-1)$. Suppose that $A$ is $R$-sectorial on $X$ of angle
$\om_{R}(A)<{\pi}(1-\frac{\alpha}{2})$.
Then for each $f\in L^p(\R_+,w_{\g};X)$ there exists a unique strong solution $u\in L^p_{\loc}([0,\infty),w_{\g};{\Do}(A))$ to \eqref{eq:fractional_evolution} which satisfies
\begin{equation}
\label{eq:fractional_space_time}
\|\partial_t^{\alpha} u\|_{L^{p}(\R_+,w_{\g };X)}+\|A u \|_{L^p(\R_+,w_{\g};X)}\lesssim \|f\|_{L^p(\R_+,w_{\g};X)}.
\end{equation}
Moreover, if $j\in \{0,1\}$ and $\alpha>j+\frac{1+\g}{p}$, then
\[\|\partial_t^j u\|_{ C_{\rm{b}}([0,\infty);\Dd_A(1-\frac{1+\g}{\alpha p}-\frac{j}{\alpha},p))}+
\|\partial_t^j u\|_{ C_{{\rm{b}},\frac{\g}{p}}(\R_+;\Dd_A(1-\frac{1}{\alpha p}-\frac{j}{\alpha},p))}\lesssim \|f\|_{L^p(\R_+,w_{\g};X)},\]
where $C_{{\rm{b}},\g/p}$ is the weighted space defined in \eqref{eq:Cbmu}.
\end{theorem}

\begin{proof}
{\em Step 1: Existence and uniqueness on $(0,T)$ with $T\in (0,\infty)$.}
Let $f\in C^\infty([0,T];\Do(A))$. By \cite[Proposition 1.2]{P93_evolutionary} there exists a unique strong solution $u\in  C([0,T];\Do(A))$ to \eqref{eq:fractional_evolution} on $[0,T]$. Then by \eqref{eq:DerTinverseid}
$u\in \Do(\Der_T^{\alpha})$ and \eqref{eq:fractionalintro} holds on $(0,T)$.
We will first show that the following estimate holds uniformly in $T>0$:
\begin{equation}\label{eq:Tfracest}
\|A u\|_{L^{p}(0,T,w_{\g};X)}\lesssim
 \|f\|_{L^{p}(0,T,w_{\g};X)}.
\end{equation}

It follows from the text below Proposition \ref{prop:Hinftyder} that $\Der_T^{\alpha}$ has a bounded $H^{\infty}$-calculus of angle $\frac{\pi}{2}\alpha$ with uniform estimates in $T$.
Since $\om_R(A)+\frac{\pi}{2}\alpha<\pi$, the Kalton--Weis theorem (see \cite[Corollary 4.5.9]{pruss2016moving}) gives \eqref{eq:Tfracest}.

From \eqref{eq:fractional_evolution}, Young's inequality and \eqref{eq:Tfracest}, we additionally obtain that
\begin{equation}
\label{eq:fractional_space_timeT}
\|u\|_{L^{p}(0,T,w_{\g };X)} \lesssim_T \|f\|_{L^p(0,T,w_{\g};X)}.
\end{equation}

Via a standard density argument using \eqref{eq:Tfracest} and \eqref{eq:fractional_space_timeT} we obtain existence and uniqueness of a strong solution to \eqref{eq:fractional_evolution} for general $f\in L^p(0,T,w_{\g};X)$. Moreover, the estimates \eqref{eq:Tfracest} and \eqref{eq:fractional_space_timeT} hold as well.

{\em Step 2. Existence and uniqueness on $(0,\infty)$ and the proof of \eqref{eq:fractional_space_time}:}
Let $f\in L^p(\R_+,w_{\gamma};X)$, and for each integer $n\geq 1$ let $u^n$ be the solution to  \eqref{eq:fractional_evolution} on $[0,n]$. Then by uniqueness, for all $n\geq m$, $u^n = u^m$ on $[0,m]$. Therefore, defining $u:[0,\infty)\to X$ by $u(t) =  u^n(t)$ for $t\in [0,n]$, we obtain a strong solution to \eqref{eq:fractional_evolution}. Uniqueness is clear from the uniqueness on finite time intervals. From Step 1 we see that \eqref{eq:Tfracest} holds with $T$-independent constants. Letting $T\to \infty$, we obtain \eqref{eq:Tfracest} for $T=\infty$.
From \eqref{eq:fractional_evolution} and \eqref{eq:inverseKalpha} we deduce that $u\in \Dd(\Der^{\alpha})$ and $\Der^{\alpha} u = f - A u$. This also implies the estimate \eqref{eq:fractional_space_time}.

{\em Step 3. Proof of the trace estimate:}
As before, by density it suffices to consider $f\in C^{\infty}_c(\R_+;\Do(A)\cap \Ran(A))$. Since $A^{-1} f\in C^\infty_c(\R_+;\Do(A))$, the above estimates also hold with $(u,f)$ replaced by $(A^{-1} u, A^{-1} f)$. By \eqref{eq:fractional_space_time} this implies $u\in L^{p}(\R_+,w_{\g};X)$ for these special functions $f$. Together with \eqref{eq:fractional_space_time} and Proposition \ref{prop:Hinftyder} this implies $u\in \Do(\Der^{\alpha})\hookrightarrow  H^{\alpha,p}(\R_+,w_{\g};X)$.
Now if $\alpha>\frac{1+\gamma}{p}+j$ with $j\in \{0,1\}$, then by Corollary \ref{c:t:trace_FF_FL_spaces;H-hom} with $\theta = 1-\frac{1+\g}{\alpha p}-\frac{j}{\alpha}$,  \eqref{eq:fractional_space_time} and Proposition \ref{prop:Hinftyder},
\begin{align*}
\|\partial_t^j u\|_{ C_{\rm{b}}([0,\infty);\Dd_A(\theta,p))} &\lesssim [\Dert^{\alpha} u]_{L^p(\R_+,w_{\g};X)}^{1-\theta} \|u\|_{L^p(\R_+,w_{\g};\Dd(A))}^{\theta}
\\ & = [\Der^{\alpha} u]_{L^p(\R_+,w_{\g};X)}^{1-\theta} \|Au\|_{L^p(\R_+,w_{\g};X)}^{\theta}
\\ & \lesssim \|f\|_{L^p(\R_+,w_{\g};X)}
\end{align*}
The $C_{{\rm{b}},\gamma/p}$-term is estimated similarly.
\end{proof}

Note that if $0\in \rho(A)$ in Theorem \ref{t:fractional}, then one obtains \[\|u\|_{H^{\alpha,p}(\R_+,w_{\g };X)} + \|u\|_{L^{p}(\R_+,w_{\g };\Do(A))}\lesssim \|f\|_{L^{p}(\R_+,w_{\g };X)}\]
and therefore $\dot{\Do}_A$ can be replaced by $\Do_A$ in the final estimate of Theorem \ref{t:fractional}. The same holds if $\R_+$ is replaced by $(0,T)$ with $T\in (0,\infty)$.

We conclude this section by analysing a double non-local diffusion equation on $\R^d$ (see \cite{KSZ17} for the case $\alpha\in (0,1)$). Setting $\beta=1$, one obtains the fractional heat equation.

\begin{example}[Double fractional diffusion equation]
\label{ex:double_fractional}
Let $\alpha,\beta>0$, $q,p\in (1,\infty)$ and $\g\in [0,p-1)$ be such that $\alpha\in(\frac{1+\g}{p},2)$. On $\R^d$ consider:
\begin{equation}
\label{eq:double_fractional_heat_equation}
\partial_t^{\alpha} u(t)+ (-\Delta)^{\beta} u(t)= f(t),\quad t>0,
\end{equation}
with Dirichlet initial condition(s). To recast the above problem in the form \eqref{eq:fractional_evolution} we introduce the fractional Laplacian.
To begin, let us regard the Laplace operator as a map
$$A_L:=-\Delta: W^{2,q}(\R^d)\subseteq L^q(\R^d)\to L^q(\R^d).$$
By \cite[Theorem 10.2.25]{Analysis2}, $A$ has a bounded $H^{\infty}$-calculus of $\angH(A_L)=0$. Therefore, $A^{\beta}_L$ has a bounded $H^{\infty}$-calculus of angle $0$ and $\Do(A_L^{\beta})=H^{2\beta,q}(\R^d)$.
In particular, by \cite[Theorem 10.3.4(2)]{Analysis2}, $A^{\beta}_L$ is $R$-sectorial and $\om_{R}(A_{L}^{\beta})=0$.
As one may expect, the operator just defined satisfies
\begin{equation}
\label{eq:fractional_Laplacian}
A_L^{\beta}f=\Four^{-1}(|\cdot|^{2\beta}\Four(f)),
\end{equation}
where $\Four$ is the Fourier transform and $f \in H^{2\beta,q}(\R^d)$ (see e.g.\ \cite[Subsection 8.3]{Haase:2}).

Combining the description of the scale $\Dd(A^{\beta}_L)$ in \cite[p. 234]{Haase:2} and the interpolation results in \cite[Theorem 6.3.1]{BeLo} we get, for all $\eta\geq 0$ and $p\in (1,\infty)$,
\begin{equation}
\label{eq:description_Dd_fractional_laplacian}
\Dd((A_L^{\beta})^{\eta})=\dot{H}^{2\beta\eta,q}(\R^d),
\quad \text{ and }\quad
\Dd_{A^{\beta}_L}(\eta,p)=\dot{B}^{2\eta\beta}_{q,p}(\R^d)
\end{equation}
where $\dot{H}$ and $\dot{B}$ denote the \textit{homogeneous} Bessel potential and Besov spaces (see e.g.\ \cite[Section 6.3]{BeLo}), respectively. Now we have rewritten \eqref{eq:double_fractional_heat_equation} in the form \eqref{eq:fractional_evolution} with $A = A^{\beta}_L$. Therefore, Theorem \ref{t:fractional} ensures that for each
$
f\in L^q(\R_+,w_{\g};L^q(\R^d))
$
there exists a unique strong solution $u\in L^p_{\loc}([0,\infty),w_{\g};H^{2\beta,q}(\R^d))$ to \eqref{eq:double_fractional_heat_equation}, and
\begin{align*}
\|\partial_t^{\alpha} u\|_{L^q(\R_+,w_{\g};L^q(\R^d))}+\|u \|_{L^q(\R_+,w_{\g};\dot{H}^{2\beta,q}(\R^d))}   \lesssim \|f\|_{L^q(\R_+,w_{\g};L^q(\R^d))}.
\end{align*}
Moreover, if $\alpha>j+\frac{1+\g}{p}$ with $j\in \{0,1\}$,  then
\begin{align*}\|\partial^j_t u\|_{C_{\rm{b}}([0,\infty);\dot{B}^{\delta}_{q,p}(\R^d))} +&
\|u\|_{C_{{\rm{b}},\gamma/p}(\R_+;\dot{B}^{\varepsilon}_{q,p}(\R^d))} \lesssim \|f\|_{L^q(\R_+,w_{\g};L^q(\R^d))}
\end{align*}
where $\delta = 2\beta(1-\frac{1+\g}{\alpha p} -\frac{j}{\alpha})$ and $\varepsilon = 2\beta(1-\frac{1}{\alpha p}-\frac{j}{\alpha})$.
\end{example}

\begin{remark}
Example \ref{ex:double_fractional} can easily be extended to the case where $\Delta$ is replaced by an elliptic operator with $x$-dependent coefficients under continuity conditions on the coefficients. Moreover, the abstract setting of Theorem \ref{t:fractional} is flexible enough to cover boundary value problems as well.
\end{remark}

\begin{remark}
A special class of Volterra type equations in the case of coefficients which are $(t,x)$-dependent is treated in \cite{DongKim20,DongKim21,DongLiu} assuming only measurability in $t$ and very weak conditions in space. They prove maximal regularity results, and their approach is completely different than the one considered here. They did not consider trace estimates, but they can easily be obtained by combining their results with ours.
\end{remark}

\subsection{Stochastic maximal $L^p$-regularity and homogeneous spaces}

In this section we prove some new regularity estimates for stochastic convolutions using the trace estimates of the previous section, which extend some of the stochastic maximal regularity results in \cite{AV20, NVWSMR} to the case where $0\notin \rho(A)$. Similar extensions can be done for the stochastic Volterra equations considered in \cite{DeschLonden13}, but we will not consider those here.

Stochastic maximal regularity has been recently used to obtain well-posedness for quasilinear stochastic PDE in \cite{AV19_QSEE_1, AV19_QSEE_2}. For unexplained notations and in particular stochastic integration theory and the theory of $\gamma$-radonifying operators we refer the reader to the latter two papers and \cite[Chapter 9]{Analysis2}. For a general overview on stochastic evolution equations we refer to \cite{DPZ}.

Let $H$ be a Hilbert space and let $W_H$ an $H$-cylindrical Brownian motion on a filtered probability space $(\O,\Fsigma,(\F_t)_{t\geq 0},\P)$.
Let $X$ be a Banach space with UMD and type $2$ (see e.g.\ \cite[Chapter 7]{Analysis2}), and let $\gamma(H,X)$ denote the space of $\gamma$-radonifying operators (see e.g.\ \cite[Chapter 9]{Analysis2}).  Given a $C_0$-semigroup $(S(t))_{t\geq 0}$ with generator $-A$, we consider the stochastic evolution equation:
\begin{equation}
\label{eq:stochastic_evolution_equation}
du(t) +A u(t) dt =G(t)dW_H(t),\qquad t\in \R_+, \ \ \  u(0) = 0.
\end{equation}
Here $G\in L^0(\Omega;L^2_{\rm loc}(\R_+;\gamma(H,X)))$ is strongly progressively measurable.
The {\em mild solution} to \eqref{eq:stochastic_evolution_equation} is given by the \emph{stochastic convolution} $S\diamond G:[0,\infty)\to X$ defined as
\begin{equation}
\label{eq:stochastic_cauchy_problem_integral}
u(t) = S\diamond G(t):=\int_0^t S(t-s)G(s)dW_H(s), \quad t\geq 0.
\end{equation}

Below we use the more suggestive notation $\partial_t^\alpha = \Der^{\alpha}$, where $\Der$ is as in \eqref{eq:defC}. The main result of this section is as follows.
\begin{theorem}
\label{t:stochastic_maximal}
Let $p\in (2,\infty)$, $\a\in [0,\frac{p}{2}-1)$ and $\theta\in [0,\frac{1}{2})$.
Let $X$ be isomorphic to a closed subspace of an $L^q$-space over a sigma-finite measure space with $q\in [2,\infty)$. Assume that $A$ has a bounded $H^{\infty}$-calculus on $X$ and $\angH(A)<\frac{\pi}{2}$. Let $(S(t))_{t\geq 0}$ be the semigroup generated by $-A$. Let $G\in L^p(\R_+\times \Omega,w_{\a};\g(H,X))$ be strongly progressive measurable. Then for all $\theta\in [0,\frac{1}{2})$
the mild solution $u$ to \eqref{eq:stochastic_evolution_equation} satisfies
\begin{equation}
\label{eq:stochastic_convolution_space_time_estimate}
\E\|\partial_t^{\theta} A^{\frac{1}{2}-\theta} u\|_{L^p(\R_+,w_{\a};X)}^p
\lesssim \E\|G\|_{L^p(\R_+,w_{\a};\g(H,X))}^p.
\end{equation}
Moreover,
\[\E\|u\|_{C_{\rm{b}}([0,\infty);\Dd_A({\frac{1}{2}-\frac{1+\a}{p},p}))}^p + \E\|u\|_{C_{{\rm{b}},\frac{\a}{p}}(\R_+;\Dd_A({\frac{1}{2}-\frac{1}{p},p}))}^p\lesssim \E\|G\|_{L^p(\R_+,w_{\a};\g(H,X))}^p,\]
where $C_{{\rm{b}},\a/p}$ is the weighted space defined in \eqref{eq:Cbmu}.
\end{theorem}

\begin{proof}
We first prove \eqref{eq:stochastic_convolution_space_time_estimate}. By an approximation argument it suffices to prove the estimates for a uniformly bounded progressively measurable step process $G: \R_+\times \O\to \g(H,\Do(A))$. The method is similar to \cite[Theorem 1.2]{NVWSMR} (and \cite[Theorem 7.16]{AV20} in the weighted case). However, since we do not assume $0\in \rho(A)$, some modifications are required.

{\em Step 1: Convergence of approximating problem.} In some part of the proof below, we need invertibility of the involved operators. Therefore, we consider $A_{\varepsilon} = A+\varepsilon$ and $S_{\varepsilon}(t) = e^{-\varepsilon t} S(t)$, where $\varepsilon\geq 0$. Let $u_{\varepsilon} = S_{\varepsilon} \diamond G$. Then $u_{\varepsilon}(t)\in \Do(A)$ a.s.\ for all $\varepsilon,t\geq 0$. Next, we prove that
\begin{equation}
\label{eq:convergence_delta_A_u_varepsilon}
 A^{\delta}_{\varepsilon}u_{\varepsilon}\to A^{\delta}u \ \text{in} \ Y_T:= L^p((0,T) \times \Omega;\Do(A)), \ \ \text{for all} \ \delta\in [0,1],\, T\in (0,\infty).
\end{equation}
Using \cite[Corollary 3.10]{NVW1} one can check that \eqref{eq:convergence_delta_A_u_varepsilon} holds for $\delta=1$.
For $\delta\in [0,1)$, note that
\begin{align}\label{eq:estAeps}
\|A^{\delta}_{\varepsilon}u_{\varepsilon} - A^{\delta}u\|_{Y_T} &\leq \|A^{\delta}_{\varepsilon} (u_{\varepsilon} - u)\|_{Y_T}  + \|A^{\delta}_{\varepsilon}u- A^{\delta}u\|_{Y_T}.
\end{align}
For the first term in \eqref{eq:estAeps}, by the moment inequality (see \cite[Proposition 6.6.4]{Haase:2}) and H\"older's inequality we find that
\[\|A^{\delta}_{\varepsilon} (u_{\varepsilon} - u)\|_{Y_T}\leq \|u_{\varepsilon} - u\|_{Y_T}^{1-\delta} \|(A+\varepsilon)(u_{\varepsilon} - u)\|_{Y_T}^{\delta}\to 0.\]
The second term in \eqref{eq:estAeps} tends to zero by \cite[Lemma 4.1.11]{Luninterp} and dominated convergence.

{\em Step 2: A priori estimates uniformly in $\varepsilon\in (0,1]$.}
Since $A$ has a bounded $H^{\infty}$-calculus of angle $\angH(A)<\frac{\pi}{2}$, the same holds for $A+\varepsilon$ with uniform estimates in $\varepsilon$. Since we assumed $X$ has a special structure, by \cite[Theorem 4.3]{NVWSMR} and \cite[Lemma 7.11]{AV20}, for all $\theta\in [0,1/2)$,
\begin{equation}
\label{eq:estimate_SMR_homongeneous}
\E\|A_{\varepsilon}^{\frac{1}{2}-\theta} S_{\varepsilon,\theta}\diamond G\|_{L^p(\R_+,w_{\a};X)}^p\lesssim \E\|G\|_{L^p(\R_+,w_{\a};\g(H,X))}^p,
\end{equation}
where $S_{\varepsilon,\theta}(t)= \frac{t^{-\theta}}{\Gamma(1-\theta)} S_{\varepsilon}(t)$ and the implicit constant does not depend on $\varepsilon$.

Next we prove \eqref{eq:stochastic_convolution_space_time_estimate} for $\theta\in (0,1/2)$ with $(u,A)$ replaced by $(u_{\varepsilon},A_{\varepsilon})$ with uniform estimates for $\varepsilon\in [0,1]$. Let $\mathscr{A}_{\varepsilon}$ be the closed linear densely defined operator on $L^p(\R_+,w_{\a};X)$ with domain $\Do(\mathscr{A}_{\varepsilon})=L^p(\R_+,w_{\a};\Do(A))$ given by $\mathscr{A}_{\varepsilon}v(t):=A_{\varepsilon}v(t)$ for all $t\in \R_+$. One can check that $\mathscr{A}_{\varepsilon}$ has a bounded $H^{\infty}$-calculus of angle $\angH(\mathscr{A}_{\varepsilon})=\angH(A_{\varepsilon})<\pi/2$ with uniform estimates in $\varepsilon\geq 0$. Moreover, by Proposition \ref{prop:Hinftyder}, the derivative $\Der$ defined in \eqref{eq:defC} has a bounded $H^\infty$-calculus of angle $\pi/2$.

From now one let $\varepsilon>0$. By \cite[Corollary 4.5.9]{pruss2016moving}, $\Co_{\varepsilon}:=\Der+\mathscr{A}_{\varepsilon}$ with $\Do(\Co_{\varepsilon}) =\Do(\Der)\cap \Do(\mathscr{A}_{\varepsilon})$ is a sectorial operator on $L^p(\R_+,w_{\a};X)$, and as in \cite{NVWSMR},
$$
\Co^{-\theta}_{\varepsilon}f(s)=\frac{1}{\Gamma(\theta)}\int_0^t (t-s)^{\theta-1}S_{\varepsilon}(t-s)f(s)ds, \ \ \ f\in L^p(\R_+,w_{\a};X).
$$
Moreover, arguing as in the proof of \cite[Theorem 1.2]{NVWSMR}, one obtains that
\begin{equation*}
 \Co^{-\theta}_{\varepsilon}(A^{\frac{1}{2}-\theta}_{\varepsilon}S_{\varepsilon,\theta}\diamond G)(t)=A^{\frac{1}{2}-\theta}_{\varepsilon}u_{\varepsilon}(t),\ \
 \text{ a.s.\ for all }t\in \R_+.
\end{equation*}
Combining this with $A^{\frac{1}{2}-\theta}_{\varepsilon}S_{\varepsilon,\theta}\diamond G, A^{\frac{1}{2}-\theta}_{\varepsilon}u_{\varepsilon}\in L^p(\R_+,w_{\a};X)$ a.s., it follows that $A^{\frac{1}{2}-\theta}_{\varepsilon}u_{\varepsilon}\in \Do(\Co^{\theta})$ a.s.\ and
\begin{equation}
\label{eq:stochastic_proof_equality}
 A^{\frac{1}{2}-\theta}_{\varepsilon}S_{\varepsilon,\theta}\diamond G(t)=\Co^{\theta}_{\varepsilon}(A^{\frac{1}{2}-\theta}_{\varepsilon}u_{\varepsilon})(t),\ \
 \text{ a.s.\ for all }t\in \R_+.
\end{equation}
To proceed further, one can check that, $\mathscr{A}_{\varepsilon}$ is $R$-sectorial of angle $<\pi/2$ with uniform estimates in $\varepsilon$ (see \cite[Theorem 10.3.4(2)]{Analysis2}). Thus, by \cite[Lemma 10]{KuWePert}, the set $\{\lambda^{\theta}(\lambda+\mathscr{A}_{\varepsilon})^{-\theta}:\lambda\in \Sigma_{\phi}\}$ is $R$-bounded for some $\phi\in (\frac{\pi}{2},\pi)$ with uniform estimates in $\varepsilon$. The Kalton-Weis theorem \cite[Theorem 4.5.6]{pruss2016moving} ensures that
\begin{equation}
\label{eq:boundedness_Co_partial}
\Der^{\theta}\Co^{-\theta}_{\varepsilon}\in \calL(L^{p}(\R_+,w_{\a};X))
\end{equation}
with uniform estimates in $\varepsilon$. Therefore,
\begin{align*}
 \E\|\Der^{\theta} A^{\frac{1}{2}-\theta}_{\varepsilon} u_{\varepsilon}\|_{L^{p}(\R_+,w_{\a};X)}^p
&\stackrel{\eqref{eq:boundedness_Co_partial}}{\lesssim} \E\|\Co^{\theta}_{\varepsilon} A^{\frac{1}{2}-\theta}_{\varepsilon} u_{\varepsilon}\|_{L^{p}(\R_+,w_{\a};X)}^p
\\
&\stackrel{\eqref{eq:stochastic_proof_equality}}{=}
\E\|A^{\frac{1}{2}-\theta}_{\varepsilon}S_{\varepsilon,\theta}\diamond G\|_{L^{p}(\R_+,w_{\a};X)}^p \\
&\stackrel{\eqref{eq:estimate_SMR_homongeneous}}{\lesssim} \E\|G\|^p_{L^{p}(\R_+,w_{\a};\g(H,X))}
\end{align*}
again with uniform estimates in $\varepsilon$.

{\em Step 3: Weak convergence argument.}
It follows from the previous step that $A^{\frac{1}{2}-\theta}_{\varepsilon} u_{\varepsilon}$ is bounded in $L^p(\Omega;\dot{\Do}(\Der^{\theta}))$, and hence there exists a sequence such that $A^{\frac{1}{2}-\theta}_{\varepsilon} u_{\varepsilon_n}\to v$ weakly in $\dot{\Do}(\Der^{\theta})$,
and
\begin{align*}
\E\|\partial_t^{\theta} v\|_{L^p(\R_+,w_{\a};X)}^p
\leq \liminf_{n\to \infty}\E\|\Der^{\theta} A^{\frac{1}{2}-\theta}_{\varepsilon_n} u_{\varepsilon_n}\|_{L^{p}(\R_+,w_{\a};X)}^p
\lesssim \E\|G\|_{L^p(\R_+,w_{\a};\g(H,X))}^p,
\end{align*}
where we used our uniform estimate in the last step.
To complete the proof of \eqref{eq:stochastic_convolution_space_time_estimate} it remains to show $v = A^{\frac{1}{2}-\theta} u$. Fix $T\in (0,\infty)$. By taking restrictions (see \eqref{eq:DerTDer}) it follows that
$A^{\frac{1}{2}-\theta}_{\varepsilon_n} u_{\varepsilon_n}\to v$ weakly in $L^p(\Omega;\dot{\Do}(\Der^{\theta}_T))$. Since $\Der^{\theta}_T$ is invertible, this implies
$A^{\frac{1}{2}-\theta}_{\varepsilon_n} u_{\varepsilon_n}\to v$ weakly in $L^p(\Omega;L^p(0,T,w_{\a};X))$. By Step 1, $v = A^{\frac{1}{2}-\theta} u$ on $(0,T)$ as required.

{\em Step 3: The trace estimate.}
By density it suffices to consider $G$ to be a uniformly bounded progressively measurable step process $G: \R_+\times \O\to \g(H,\Do(A)\cap \Ran(A))$. Note that by the previous step $A^{1/2} u\in L^p(\Omega;L^p(\R_+,w_{\a};X))$. Moreover, due to the fact that $G$ also takes value in $\Do(A^{-\sigma})$  for all $\sigma\in [0,1/2)$ (see \cite[Proposition 15.26]{KuWe}), we additionally have
\begin{equation}\label{eq:uLpaswell}
 A^{\frac12-\sigma} u = A^{\frac12}  A^{-\sigma}u\in L^p(\Omega;L^p(\R_+,w_{\a};X)).
\end{equation}
Fix $\sigma\in (\frac{1+\a}{p},\frac{1}{2})$.
By \eqref{eq:stochastic_convolution_space_time_estimate} we have
\begin{align}\label{eq:reguSEE}
\partial_t^{\sigma} u \in L^p(\O; L^p(\R_+,w_{\a};\Dd(A^{\frac{1}{2}-\sigma}))) \ \ \text{and} \ \
 u \in L^p(\O; L^p(\R_+,w_{\a};\Dd(A^{\frac{1}{2}})))
\end{align}
with corresponding estimates in terms of $\|G\|^p_{L^p(\Omega;L^{p}(\R_+,w_{\a};\g(H,X)))}$. Moreover, combining the latter with \eqref{eq:uLpaswell} and Proposition \ref{prop:Hinftyder}, we obtain that
\begin{align*}
u & \in  L^p(\O;\Hz^{\sigma,p}(\R_+,w_{\a};\Dd(A^{\frac{1}{2}-\sigma})))
  \hookrightarrow L^p(\O;H^{\sigma,p}(\R_+,w_{\a};\Dd(A^{\frac{1}{2}-\sigma}))).
\end{align*}
Applying, Corollary \ref{c:t:trace_FF_FL_spaces;H-hom} pointwise in $\O$, and using  \eqref{eq:reguSEE} and \[(\Dd(A^{\frac{1}{2}-\sigma}),\Dd(A^{\frac{1}{2}}))_{1-\frac{1+\a}{p\sigma},p} \stackrel{\eqref{eq:Dd_A_reiteration}}{=} \Dd_{A}(\tfrac{1}{2}-\tfrac{1+\a}{p},p),
\]
one gets the required estimates for the $C_{\rm{b}}$-norm. The estimate for the $C_{{\rm{b}},\a/p}$-norm follows similarly.
\end{proof}

If $0\in \rho(A)$ in Theorem \ref{t:stochastic_maximal}, then \eqref{eq:stochastic_convolution_space_time_estimate} can be improved into
\begin{equation*}
\E\|S\diamond G\|_{H^{\theta,p}(\R_+,w_{\a};\Do(A^{\frac{1}{2}-\theta}))}^p
\lesssim \E\|G\|_{L^p(\R_+,w_{\a};\g(H,X))}^p,
\end{equation*}
and therefore $\dot{\Do}_A$ can be replaced by $\Do_A$ in the final estimate of Theorem \ref{t:stochastic_maximal}. The same holds if one considers $(0,T)$ with $T\in (0,\infty)$ instead of $\R_+$.

As an application we derive a \emph{new} maximal estimate for the stochastic heat equation using homogenous spaces.
In the following, $(w_t^n:t\geq 0)_{n\geq 1}$ denotes a sequence of independent standard Brownian motions. This determines a cylindrical Brownian motion on $\ell^2$. We will use that $\gamma(\ell^2;L^q(\R^d)) = L^q(\R^d;\ell^2)$ (see \cite[Proposition 2.6]{NVW1}).

\begin{example}[Stochastic heat equation on $\R^d$]
\label{ex:stochastic_fractional}
Let $q\in [2,\infty)$, $p\in (2,\infty)$ and $\a\in [0,\frac{p}{2}-1)$. The stochastic heat equation on $\R^d$ can be formally written as
\begin{equation}
\label{eq:fractional_heat_stochastic}
\begin{cases}
du -\Delta u dt=\sum_{n\geq 1} g_n dw^n_t,\qquad t>0,\\
u(0)=0.
\end{cases}
\end{equation}
Using the notation introduced in Example \ref{ex:double_fractional}, we recast \eqref{eq:fractional_heat_stochastic} in the form \eqref{eq:stochastic_evolution_equation} with $A = -\Delta$ on $L^q(\R^d)$. Then $A$ has a bounded $H^{\infty}$-calculus of angle $0$ and $\Dd(A^{\eta})$ and $\Dd_{A}(\eta,q)$ are a special case of \eqref{eq:description_Dd_fractional_laplacian} for $\beta=1$.
Thus, Theorem \ref{t:stochastic_maximal} ensures that for any progressively measurable $(g_n)_{n\geq 1}\in L^p( \O;L^p(\R_+,w_{\a};L^q(\R^d;\ell^2)))$
there exists a unique mild solution $u$ to \eqref{eq:fractional_heat_stochastic}.
Moreover, for all $\theta\in [0,1/2)$,
\begin{align*}
\E\|\partial_t^{\theta} u&\|_{L^p(\R_+,w_{\a};\dot{H}^{1-2\theta,q}(\R^d))}^p  +\E\|u\|_{C_{\rm{b}}([0,\infty);\dot{B}^{\alpha}_{q,p}(\R^d))}^p \\ & +  \E\|u\|_{C_{{\rm{b}},\a/p}(\R_+;\dot{B}^{\beta}_{q,p}(\R^d))}^p \lesssim \E \|(g_n)_{n\geq 1}\|_{L^p(\R_+,w_{\a};L^q(\R^d;\ell^2))}^p,
\end{align*}
where $\alpha= 1-2\frac{1+\a}{p}$ and $\beta = 1-\frac{2}{p}$.
\end{example}

\def\polhk#1{\setbox0=\hbox{#1}{\ooalign{\hidewidth
  \lower1.5ex\hbox{`}\hidewidth\crcr\unhbox0}}} \def\cprime{$'$}


\begin{thebibliography}{10}

\bibitem{AV19_QSEE_1}
A.~Agresti and M.~Veraar.
\newblock Nonlinear parabolic stochastic evolution equations in critical spaces
  {P}art {I}. {S}tochastic maximal regularity and local existence.
\newblock To appear in Nonlinearity. arXiv preprint arXiv:2001.00512, 2020.

\bibitem{AV19_QSEE_2}
A.~Agresti and M.~Veraar.
\newblock Nonlinear parabolic stochastic evolution equations in critical spaces
  {P}art {II}. {B}low criteria and instantaneous regularization.
\newblock To appear in J. Evol. Equ. arXiv preprint arXiv:2012.04448, 2020.

\bibitem{AV20}
A.~Agresti and M.~Veraar.
\newblock Stability properties of stochastic maximal {$L^p$}-regularity.
\newblock {\em J. Math. Anal. Appl.}, 482(2):123553, 35, 2020.

\bibitem{Amannbook2}
H.~Amann.
\newblock {\em {Linear and quasilinear parabolic problems. Volume II: Function
  spaces.}}, volume 106.
\newblock Cham: Birkh\"auser, 2019.

\bibitem{Bazhlekova_2001_dissertation}
E.G. Bazhlekova.
\newblock {\em Fractional evolution equations in Banach spaces}.
\newblock PhD thesis, Department of Mathematics and Computer Science, 2001.

\bibitem{BeLo}
J.~Bergh and J.~L{\"o}fstr{\"o}m.
\newblock {\em Interpolation spaces. {A}n introduction}.
\newblock Springer-Verlag, Berlin, 1976.
\newblock Grundlehren der Mathematischen Wissenschaften, No. 223.

\bibitem{Bui82}
H.-Q. Bui.
\newblock Weighted {B}esov and {T}riebel spaces: interpolation by the real
  method.
\newblock {\em Hiroshima Math. J.}, 12(3):581--605, 1982.

\bibitem{CLS04}
P.~Cl\'{e}ment, S.-O. Londen, and G.~Simonett.
\newblock Quasilinear evolutionary equations and continuous interpolation
  spaces.
\newblock {\em J. Differential Equations}, 196(2), 2004.

\bibitem{DPZ}
G.~Da~Prato and J.~Zabczyk.
\newblock {\em Stochastic equations in infinite dimensions}, volume~44 of {\em
  Encyclopedia of Mathematics and its Applications}.
\newblock Cambridge University Press, Cambridge, 1992.

\bibitem{danchin2020free}
R.~Danchin, M.~Hieber, P.~B Mucha, and P.~Tolksdorf.
\newblock Free boundary problems via da prato-grisvard theory.
\newblock {\em arXiv preprint arXiv:2011.07918}, 2020.

\bibitem{DHPinh}
R.~Denk, M.~Hieber, and J.~Pr\"uss.
\newblock Optimal {$L^p$}-{$L^q$}-estimates for parabolic boundary value
  problems with inhomogeneous data.
\newblock {\em Math. Z.}, 257(1):193--224, 2007.

\bibitem{DenkKaip}
R.~Denk and M.~Kaip.
\newblock {\em General parabolic mixed order systems in {${L_p}$} and
  applications}, volume 239 of {\em Operator Theory: Advances and
  Applications}.
\newblock Birkh\"{a}user/Springer, Cham, 2013.

\bibitem{DenkSaalSeiler}
R.~Denk, J.~Saal, and J.~Seiler.
\newblock Inhomogeneous symbols, the {N}ewton polygon, and maximal
  {$L^p$}-regularity.
\newblock {\em Russ. J. Math. Phys.}, 15(2):171--191, 2008.

\bibitem{DeschLonden13}
G.~Desch and S.-O. Londen.
\newblock Maximal regularity for stochastic integral equations.
\newblock {\em J. Appl. Anal.}, 19(1):125--140, 2013.

\bibitem{DiBlasio84}
G.~Di~Blasio.
\newblock Linear parabolic evolution equations in {$L^p$}-spaces.
\newblock {\em Ann. Mat. Pura Appl. (4)}, 138:55--104, 1984.

\bibitem{DongKim20}
H.~Dong and D.~Kim.
\newblock {$L_p$}-estimates for time fractional parabolic equations in
  divergence form with measurable coefficients.
\newblock {\em J. Funct. Anal.}, 278(3):108338, 66, 2020.

\bibitem{DongKim21}
H.~Dong and D.~Kim.
\newblock An approach for weighted mixed-norm estimates for parabolic equations
  with local and non-local time derivatives.
\newblock {\em Adv. Math.}, 377:107494, 44, 2021.

\bibitem{DongLiu}
H.~Dong and Y.~Liu.
\newblock Weighted mixed norm estimates for fractional wave equations with vmo
  coefficients.
\newblock {\em arXiv preprint arXiv:2102.01136}, 2021.

\bibitem{F90_1}
Y.~Fujita.
\newblock Integrodifferential equation which interpolates the heat equation and
  the wave equation.
\newblock {\em Osaka J. Math.}, 27(2):309--321, 1990.

\bibitem{GraModern}
L.~Grafakos.
\newblock {\em Modern {F}ourier analysis}, volume 250 of {\em Graduate Texts in
  Mathematics}.
\newblock Springer, New York, second edition, 2009.

\bibitem{GLS}
G.~Gripenberg, S.-O. Londen, and O.~Staffans.
\newblock {\em Volterra integral and functional equations}, volume~34 of {\em
  Encyclopedia of Mathematics and its Applications}.
\newblock Cambridge University Press, Cambridge, 1990.

\bibitem{Haase:2}
M.H.A. Haase.
\newblock {\em The functional calculus for sectorial operators}, volume 169 of
  {\em Operator Theory: Advances and Applications}.
\newblock Birkh\"auser Verlag, Basel, 2006.

\bibitem{HanHyt}
T.S. H\"{a}nninen and T.P. Hyt\"{o}nen.
\newblock The {$A_2$} theorem and the local oscillation decomposition for
  {B}anach space valued functions.
\newblock {\em J. Operator Theory}, 72(1):193--218, 2014.

\bibitem{H00}
R.~Hilfer.
\newblock Fractional time evolution.
\newblock In {\em Applications of fractional calculus in physics}, pages
  87--130. World Sci. Publ., River Edge, NJ, 2000.

\bibitem{Analysis1}
T.P. Hyt\"onen, J.M.A.M.~van Neerven, M.C. Veraar, and L.~Weis.
\newblock {\em Analysis in {B}anach spaces. {V}ol. {I}. {M}artingales and
  {L}ittlewood-{P}aley theory}, volume~63 of {\em Ergebnisse der Mathematik und
  ihrer Grenzgebiete. 3. Folge.}
\newblock Springer, 2016.

\bibitem{Analysis2}
T.P. Hyt\"onen, J.M.A.M.~van Neerven, M.C. Veraar, and L.~Weis.
\newblock {\em Analysis in {B}anach spaces. {V}ol. {II}. {P}robabilistic
  {M}ethods and {O}perator {T}heory.}, volume~67 of {\em Ergebnisse der
  Mathematik und ihrer Grenzgebiete. 3. Folge.}
\newblock Springer, 2017.

\bibitem{KSVZ16}
J.~Kemppainen, J.~Siljander, V.~Vergara, and R.~Zacher.
\newblock Decay estimates for time-fractional and other non-local in time
  subdiffusion equations in {$\Bbb{R}^d$}.
\newblock {\em Math. Ann.}, 366(3-4):941--979, 2016.

\bibitem{KSZ17}
J.~Kemppainen, J.~Siljander, and R.~Zacher.
\newblock Representation of solutions and large-time behavior for fully
  nonlocal diffusion equations.
\newblock {\em J. Differential Equations}, 263(1):149--201, 2017.

\bibitem{Kree}
P.~Kr\'{e}e.
\newblock Sur les multiplicateurs dans {${\mathcal{F}} L^{p}$} avec poids.
\newblock {\em Ann. Inst. Fourier (Grenoble)}, 16:91--121, 1966.

\bibitem{KuWePert}
P.~C. Kunstmann and L.W. Weis.
\newblock Perturbation theorems for maximal {${L}_p$}-regularity.
\newblock {\em Ann. Scuola Norm. Sup. Pisa Cl. Sci. (4)}, 30(2):415--435, 2001.

\bibitem{KuWe}
P.C. Kunstmann and L.W. Weis.
\newblock Maximal {$L\sb p$}-regularity for parabolic equations, {F}ourier
  multiplier theorems and {$H\sp \infty$}-functional calculus.
\newblock In {\em Functional analytic methods for evolution equations}, volume
  1855 of {\em Lecture Notes in Math.}, pages 65--311. Springer, Berlin, 2004.

\bibitem{LSU}
O.~A. Lady\v{z}enskaja, V.~A. Solonnikov, and N.~N. Ural\cprime~ceva.
\newblock {\em Linear and quasilinear equations of parabolic type}.
\newblock Translated from the Russian by S. Smith. Translations of Mathematical
  Monographs, Vol. 23. American Mathematical Society, Providence, R.I., 1968.

\bibitem{Lindinter}
N.~Lindemulder.
\newblock An intersection representation for a class of anisotropic
  vector-valued function spaces.
\newblock {\em J. Approx. Theory}, 264:105519, 2021.

\bibitem{LMV18}
N.~Lindemulder, M.~Meyries, and M.~Veraar.
\newblock Complex interpolation with {D}irichlet boundary conditions on the
  half line.
\newblock {\em Math. Nachr.}, 291(16):2435--2456, 2018.

\bibitem{Luninterp}
A.~Lunardi.
\newblock {\em Interpolation {T}heory}.
\newblock Appunti. Scuola Normale Superiore Pisa, 1999.

\bibitem{MartinezSanz}
C.~Mart\'{\i}nez~Carracedo and M.~Sanz~Alix.
\newblock {\em The theory of fractional powers of operators}, volume 187 of
  {\em North-Holland Mathematics Studies}.
\newblock North-Holland Publishing Co., Amsterdam, 2001.

\bibitem{MeySchn12}
M.~Meyries and R.~Schnaubelt.
\newblock Interpolation, embeddings and traces of anisotropic fractional
  {S}obolev spaces with temporal weights.
\newblock {\em J. Funct. Anal.}, 262(3):1200--1229, 2012.

\bibitem{MeySchnau12b}
M.~Meyries and R.~Schnaubelt.
\newblock Maximal regularity with temporal weights for parabolic problems with
  inhomogeneous boundary conditions.
\newblock {\em Math. Nachr.}, 285(8-9):1032--1051, 2012.

\bibitem{MV12}
M.~Meyries and M.C. Veraar.
\newblock Sharp embedding results for spaces of smooth functions with power
  weights.
\newblock {\em Studia Math.}, 208(3):257--293, 2012.

\bibitem{MV14}
M.~Meyries and M.C. Veraar.
\newblock Traces and embeddings of anisotropic function spaces.
\newblock {\em Math. Ann.}, 360(3-4):571--606, 2014.

\bibitem{NVW1}
J.M.A.M.~van Neerven, M.C. Veraar, and L.W. Weis.
\newblock Stochastic integration in {UMD} {B}anach spaces.
\newblock {\em Ann. Probab.}, 35(4):1438--1478, 2007.

\bibitem{NVWSMR}
J.M.A.M.~van Neerven, M.C. Veraar, and L.W. Weis.
\newblock Stochastic maximal {$L^p$}-regularity.
\newblock {\em Ann. Probab.}, 40(2):788--812, 2012.

\bibitem{P93_evolutionary}
J.~Pr\"{u}ss.
\newblock {\em Evolutionary integral equations and applications}.
\newblock Modern Birkh\"{a}user Classics. Birkh\"{a}user/Springer Basel AG,
  Basel, 1993.
\newblock [2012] reprint of the 1993 edition.

\bibitem{Pruss19}
J.~Pr\"{u}ss.
\newblock Vector-valued {F}ourier multipliers in {$L_p$}-spaces with power
  weights.
\newblock {\em Studia Math.}, 247(2):155--173, 2019.

\bibitem{PruSim04}
J.~Pr{\"u}ss and G.~Simonett.
\newblock {Maximal regularity for evolution equations in weighted
  $L_p$-spaces}.
\newblock {\em Archiv der Mathematik}, 82(5):415--431, 2004.

\bibitem{pruss2016moving}
J.~Pr{\"u}ss and G.~Simonett.
\newblock {\em Moving interfaces and quasilinear parabolic evolution
  equations}, volume 105.
\newblock Springer, 2016.

\bibitem{CriticalQuasilinear}
J.~Pr\"{u}ss, G.~Simonett, and M.~Wilke.
\newblock Critical spaces for quasilinear parabolic evolution equations and
  applications.
\newblock {\em J. Differential Equations}, 264(3):2028--2074, 2018.

\bibitem{Rychkov1999}
V.S. Rychkov.
\newblock On restrictions and extensions of the {B}esov and
  {T}riebel-{L}izorkin spaces with respect to {L}ipschitz domains.
\newblock {\em J. London Math. Soc. (2)}, 60(1):237--257, 1999.

\bibitem{Saw18}
Y.~Sawano.
\newblock {\em Theory of {B}esov spaces}, volume~56 of {\em Developments in
  Mathematics}.
\newblock Springer, Singapore, 2018.

\bibitem{SS04}
H.-J. Schmei{\ss}er and W.~Sickel.
\newblock Traces, {G}agliardo-{N}irenberg {I}nequalities and {S}obolev {T}ype
  {E}mbeddings for {V}ector-valued {F}unction {S}paces.
\newblock Unpublished notes, Jena, 2004.

\bibitem{stein1957note}
E.~M. Stein.
\newblock Note on singular integrals.
\newblock {\em Proceedings of the American Mathematical Society},
  8(2):250--254, 1957.

\bibitem{Stein93}
E.M. Stein.
\newblock {\em Harmonic analysis: real-variable methods, orthogonality, and
  oscillatory integrals}, volume~43 of {\em Princeton Mathematical Series,
  Monographs in Harmonic Analysis, III}.
\newblock Princeton University Press, Princeton, NJ, 1993.

\bibitem{Tr78}
H.~Triebel.
\newblock {\em Interpolation theory, function spaces, differential operators},
  volume~18 of {\em North-Holland Mathematical Library}.
\newblock North-Holland Publishing Co., Amsterdam-New York, 1978.

\bibitem{Tri83}
H.~Triebel.
\newblock {\em Theory of function spaces}, volume~78 of {\em Monographs in
  Mathematics}.
\newblock Birkh\"auser Verlag, Basel, 1983.

\bibitem{Tri92II}
H.~Triebel.
\newblock {\em Theory of function spaces. {II}}, volume~84 of {\em Monographs
  in Mathematics}.
\newblock Birkh\"{a}user Verlag, Basel, 1992.

\bibitem{Tr97}
H.~Triebel.
\newblock {\em Fractals and spectra}, volume~91 of {\em Monographs in
  Mathematics}.
\newblock Birk\-h\"auser Verlag, Basel, 1997.
\newblock Related to Fourier analysis and function spaces.

\bibitem{Tri06III}
H.~Triebel.
\newblock {\em {T}heory of function spaces. {III}}, volume 100 of {\em
  Monographs in Mathematics}.
\newblock Birkh\"{a}user Verlag, Basel, 2006.

\bibitem{Zacher05}
R.~Zacher.
\newblock Maximal regularity of type {$L\sb p$} for abstract parabolic
  {V}olterra equations.
\newblock {\em J. Evol. Equ.}, 5(1):79--103, 2005.

\end{thebibliography}
\end{document}